\pdfoutput=1

\documentclass[a4paper,11pt]{amsart}
\usepackage{amsmath,amscd,amsfonts,amssymb,latexsym}
\usepackage[cmtip,arrow]{xy}
\usepackage{pb-diagram,pb-xy}
\usepackage{longtable}
\usepackage{multirow,tabularx}
\usepackage{enumitem}

\newtheorem{theorem}{Theorem}
\newtheorem{proposition}[theorem]{Proposition}

\newtheorem{corollary}[theorem]{Corollary}
\newtheorem{remark}[theorem]{Remark}

\def\C{\mathbb{C}}
\def\Q{\mathbb{Q}}
\def\T{{\mathbb{C}^*}}
\def\TT{\mathbb{T}}
\def\R{\mathbb{R}}
\def\Z{\mathbb{Z}}
\def\P{\mathbb{P}}
\def\qed{\hfill$\Box$}

\def\codim{{\rm{codim}}}
\def\OOO{{\mathcal  O}}
\def\TU{{\widetilde T}}%{\mathcal{T}}}
\def\eps{\varepsilon}
\begin{document}

\title[Hirzebruch vs Molien]
{Equivariant Hirzebruch classes and Molien series of quotient singularities}
\author{Maria Donten-Bury \& Andrzej Weber}\thanks{Supported by NCN grant 2013/08/A/ST1/00804.\\This work was completed while MD-B held a Dahlem Research School Postdoctoral Fellowship at Freie Universit\"at Berlin.}

\address{University of Warsaw, Institute of Mathematics, Banacha 2, 02-097 Warszawa, Poland \& Freie Universit\"at Berlin, Institute of Mathematics, Arnimallee 3, 14195 Berlin, Germany}

\email{m.donten@mimuw.edu.pl}

\address{University of Warsaw, Institute of Mathematics, Banacha 2, 02-097 Warszawa, Poland \&
 Institute of Mathematics\\ Polish Academy of Sciences\\ ul. \'Sniadeckich 8, 00-656 Warszawa, Poland}

 \email{aweber@mimuw.edu.pl}

\begin{abstract}
We study properties of the Hirzebruch class of quotient singularities $\C^n/G$, where $G$ is a finite matrix group. The main result states that the Hirzebruch class coincides with the Molien series of $G$ under suitable substitution of variables. The Hirzebruch class of a crepant resolution can be described specializing the orbifold elliptic genus constructed by Borisov and Libgober. It is equal to the combination of Molien series of centralizers of elements of $G$. This is an incarnation of the McKay correspondence.
The results are illustrated with several examples, in particular of 4-dimensional symplectic quotient singularities.
\end{abstract}
\maketitle

\section{Introduction}

The McKay correspondence is a postulated relation between the geometry of a crepant resolution of a quotient singularity and the properties of the group defining the singularity, or its representation theory, see~\cite{ItoReid, ReidMcKay}. A resolution of singularities $\pi \colon Y \rightarrow X$ is called crepant if $\varphi^*K_X = K_Y$. One may think that this condition is to ensure that the resolution is, in some sense, ``not too big'', it does not have unnecessary components. It was first observed by McKay in the 2-dimensional case,~\cite{McKay}, that the structure of the minimal resolution of a Du Val singularity, $\C^2/G$ for $G \subset SL(2,\C)$, can be described in terms of the group structure: the components of the exceptional fiber correspond to conjugacy classes of elements of~$G$, or to its irreducible representations. Reid conjectured, see~\cite{ReidMcKay}, that this should be true in more general setting. Since then the correspondence was proven in certain cases on various levels of detail. First came the proof in dimension~3, given e.g. in~\cite{ItoReid}. The weak version, i.e. the equality between~$\dim H^*(E)$, where $E$ is the exceptional divisor, and the number of conjugacy classes in $G$ is due to Batyrev for any $G \subset SL(n,\C)$, \cite{Bat}. For symplectic resolutions of symplectic quotient singularities Kaledin shown that there is a natural bijection between conjugacy classes in~$G$ and the basis of cohomology, \cite{KaledinMcKay}. The correspondence conjecture can be also rephrased in the language of derived categories, see~\cite{BridgelandKingReid}.
\smallskip

There is a broader understanding of the McKay correspondence. We search for a relation between
the geometric properties of the resolution  and the algebraic properties of the group and its representation. Here we do not assume that the resolution is crepant. It can be any resolution, but we pay a price: we have to correct the data if the resolution is too big. This is the strategy of \cite{BoLi} or \cite{Ve}, also present in \cite{Bat}. The correction terms depend on the structure of the resolution. The invariant which we are interested in is the {\it Hirzebruch class}  $$td_y(X)\in H_*(X)\otimes\Q[y]\,.$$
It is defined both for the quotient variety, which is singular in general, and its resolution.
If $X$ admits a crepant resolution $f:\widetilde X\to X$ then $td_{y=0}(X)$ coincides with the image of the classical Todd class of the resolution $$td_{y=0}(X)=f_*td(\widetilde X)\,.$$
The equality does not hold for the full Hirzebruch class, as it seen already by the example of Du Val singularities. However, the difference between $td_{y}(X\hookrightarrow M)$ and $f_*td_{y}(\widetilde X)$ is well controlled.

 It is convenient to assume that $X$ is embedded in an ambient smooth variety $M$ and to study the image $$td_y(X\hookrightarrow M)\in H_*(M)\otimes\Q[y]\simeq H^*(M)\otimes\Q[y].$$ Since we are interested in equivariant situation with respect to the torus $\T$ action, our  enriched invariant belongs to the equivariant cohomology $$\hat H^*_\T(X)\otimes\Q[y]=\left(\prod_{k=0}^\infty H^k_\T(X)\right)\otimes\Q[y]\,.$$
To compute the Hirzebruch class we use its functorial and motivic properties \cite[Theorem 3.1]{BSY}, see also \cite[\S5]{Wbb}:
\begin{equation}\label{tdy-indukcyjnie}td_y(X\hookrightarrow M)=p_*td_y(\widetilde X)-p_*td_y(E\hookrightarrow \widetilde X)+td_y(p_*(E)\hookrightarrow M)\,.\end{equation}
Here $p:\widetilde X\to X$ is the resolution of singularities and $E\subset \widetilde X$ is the exceptional divisor.
The formula (\ref{tdy-indukcyjnie}) can be treated as an inductive definition of the Hirzebruch class for singular varieties, provided that we know what $td_y( X)$ is for a smooth variety as the initial step of the induction. In the smooth case the class $td_y(X)=td(TX)ch(\Lambda_y(T^*X))$ was defined by Hirzebruch  \cite[Chapter 4]{Hir}. (We recommend \cite[\S5.1]{Huy} as a short introduction.)
It is the multiplicative characteristic class associated to the formal power series
$$h_y(x)=x\frac{1+y e^{-x}}{1-e^{-x}}\,.$$
The Hirzebruch class for singular varieties was defined in \cite{BSY}. The equivariant version is studied in \cite{We3}, using the method of \cite{Oh}, see also the introductions to \cite{MiWe,Wbb}.
 We wish to study singularities locally, therefore we restrict the Hirzebruch class to the singular point, or equivalently to any contractible $\T$-stable neighbourhood of the singular point. We will assume that the ambient space  $M$ is a vector space with a linear action of $\T$.
 This way we obtain the local equivariant Hirzebruch class which belongs to the equivariant cohomology
$$\hat H^*_\T(M)\otimes \Q[y]\simeq\hat H^*_\T(pt)\otimes \Q[y]\simeq \Q[[t]][y]\,.$$
 One can express the Hirzebruch class in a convenient way by taking $T=e^{-t}$ and using the Euler class of the $\T$-representation on the tangent space $T_0M\simeq M$. Our first result Theorem \ref{th1} says that the Hirzebruch class of the quotient $M/G$ essentially coincides with the (extended) Molien series of the representation of $G$ on $M\simeq\C^n$.

A fundamental problem of the McKay correspondence is to compute invariants of a crepant resolution of a quotient singularity in terms of invariants of the action of $G$. In our case we consider quotients of an affine space, but to compute the Hirzebruch class of a crepant resolution (if it exists) we recall a construction which holds in general, and moreover is valid in the equivariant setup. The general theory is built in a series of papers of Borisov and Libgober \cite{BoLi0, BoLi1, BoLi}. They define a more delicate invariant, called the elliptic genus, and a related cohomology class. It is not motivic, i.e.~it does not behave additively with respect to cut and paste operations. Nevertheless, the elliptic class can be defined for certain class of singular varieties, including quotient singularities.

Due to \cite{BoLi} the elliptic class of $X/G$ can be computed in terms of various data associated to the fixed point sets $X^g$ for $g\in G$. If $X/G$ has a crepant resolution, then the elliptic class of $X/G$ is the image of the elliptic class of the resolution. If $X=V$ is an affine space we obtain a rather complicated description of the equivariant elliptic class of $V/G$ in purely algebraic terms. The elliptic class (depending on a formal parameter $q$) specializes to the Hirzebruch class when $q$ tends to $0$. This way, somehow going around, we arrive at the formula for the Hirzebruch class of a crepant resolution of $V/G$. This time the formula is easy. We obtain a combination of Molien series of centralizers of elements of $G$. The formula makes sense even when a crepant resolution does not exist.

\medskip
The results presented here fit in the general idea of describing the geometry of  resolutions of a quotient singularity in terms of group properties. In Theorem~\ref{th1} we show that the equivariant Hirzebruch class of a quotient singularity $\C^n/G$ is very closely related to (and can be easily derived from) the (extended) Molien series of~$G$, containing information about the degrees of invariants of exterior and symmetric powers of the considered representation of~$G$.
For us the symplectic actions are the most interesting. Their Hirzebruch class has a very particular form. We have computed multiple examples (mainly in dimension 4) and we have noticed certain phenomena, which deserve further study.

Maybe the most intriguing one is certain positivity of the local equivariant Hirzebruch class. The search of positivity of singular characteristic classes was started by Aluffi and Mihalcea in \cite{AlMi}. They observed that Chern-Schwartz-MacPherson classes of Schubert cells in the classical Grassmannians are effective in all computed cases. Being effective (in that case) is equivalent to having nonnegative intersections with the basis of the cohomology of Grassmannian formed by Schubert classes. It was conjectured that positivity holds for Grassmannians. The conjecture was proven by Huh \cite{Huh}. His proof works only for classical Grassmannians, while the question of positivity makes sense for any homogeneous space $G/P$. Also, the proof of Huh cannot be repeated in the equivariant setting.

Positivity of local equivariant Chern-Schwartz-MacPherson classes for Schubert varieties was noticed in \cite{We1} by computer experiments. So far there is no proof. Moreover, in \cite{We3}, it was noticed that there is another, stronger positivity of local equivariant Hirzebruch classes. Positivity was proven for simplicial toric varieties, while for various Schubert cells in $G/P$ it was only observed in the results of computations. The positivity of local equivariant Hirzebruch classes implies the positivity of local Chern-Schwartz-MacPherson classes. In the present paper we prove 
a particular form of positivity of local equivariant Hirzebruch classes as well as positivity of local equivariant Chern-Schwartz-MacPherson classes of quotient varieties $\C^n/G$. 

We would like to remark that we do not know what are the meanings and consequences of both positivities. It is just a phenomenon, which we observe and prove in some cases. We wish to have applications and to find relations with other properties of singularities.

\bigskip

The results of the paper are the following:
\begin{itemize}
\item Theorem \ref{th1} shows that essentially the equivariant Hirzebruch class of $\C^n/G$ is equal to the Molien  series of the representation. The proof is based on a version of the Lefschetz-Riemann-Roch theorem proved in \cite{CMSS}. It is given in Section \ref{secLRR}.\medskip

\item In Section \ref{section_functional_eq} there are described symmetries of the Hirzebruch class of the quotient singularity.
   The proof is based on the interpretation of the equivariant Hirzebruch class  as the Molien series. This is done in  Section~\ref{section_functional_eq}.
    For $G\subset SL_n(\C)$ or symplectic quotients we have additional functional equations.\medskip

\item We illustrate Theorem \ref{th1} by the example of Du Val singularities in Section \ref{sectionduval}.\medskip

\item Theorem \ref{bltheorem} and Corollary
\ref{blcorollary} in Section \ref{sec-crepant}  describe the equivariant Hirzebruch class of a crepant resolution in terms of Molien series. The result is a specialization of the McKay correspondence proved for elliptic class by Borisov and Libgober \cite{BoLi}.\medskip

\item In Section \ref{final}  we discuss positivity of the local equivariant Hirzebruch class.

\item In Appendix \ref{sectionexamples} we give a series of examples of Hirzebruch classes for symplectic quotients.\medskip

\end{itemize}

The symplectic singularities were in the center of our interest. Their crepant resolutions are automatically Hyperk\"ahler. For this class of singularities  probably one should define and study a characteristic class which would contain more information than the elliptic class. We leave this subject for future research.

\section{The main result}

Suppose $G\subset GL_n(\C)$ is a finite group.
In general the quotient $X=\C^n/G$ is a singular variety which admits an action of $\C^*$ coming from scalar multiplication in $\C^n$.
Suppose~$X$ is embedded equivariantly in a vector space~$M$.
As $M^*$ we can take the space freely spanned by the generators of $(Sym^\bullet(\C^n))^G$. Precisely, the set of homogeneous generators $s_1,s_2,\dots,s_m\in(Sym^\bullet(\C^n))^G$ defines a surjection from the polynomial ring $\C[\tilde s_1,\tilde s_2,\dots,\tilde s_m]\twoheadrightarrow (Sym^\bullet(\C^n))^G$. This defines an embedding $\C^n/G\hookrightarrow M=\C^m$. The vector space $M$ admits a linear action of $\T$ such that the embedding is equivariant. The weights of this action are $w_k=\deg(s_k)$ for $k=1,2,\dots,m$.
The equivariant Hirzebruch class (see e.g.~\cite{BSY, We3})
$$td^{\C^*}_y(X\hookrightarrow M)\in \hat H^{\C^*}_*(M)\otimes \Q[y]\simeq \Q[[t]][y]$$
is of the form
$$eu(M,0)\cdot H(y,e^{-t})\,,$$
where \begin{itemize}
\item $eu(M,0)$ is the Euler class at the fixed point $p$, that is
$$eu(M,0)=\left(\prod_{k=1}^{\dim( M)} w_k\right) t^{\dim (M)}\,,$$
is the product of  the weights  $w_k\in\Z$ of the action of $\C^*$ on $T_0M\simeq M$.

\item $H(y,T)$ is a rational function in $T=e^{-t}$. When multiplied by $$\prod_{k=1}^{\dim (M)}(1-T^{w_k})$$ it is a polynomial in $T$ and $y$. The function $H(y,T)$ does not depend on the embedding.
\end{itemize}
We stress that the equivariant Hirzebruch class is an invariant of a singularity computed via resolution. That is so in general, but of course in some particular cases one can avoid resolutions -- see our computation of the Hirzebruch class for Du Val singularities given in~\S\ref{duval}. \smallskip

On the other hand, we have a purely algebraic invariant of the representation of $G$. The (extended) Molien series is defined by the formula:
$$Mol(v,T)=\frac{1}{|G|}\sum_{k=0}^\infty\sum_{\ell=0}^n \dim((\Lambda^\ell(\C^n)^*\otimes Sym^k(\C^n)^*)^G)v^\ell T^k\,.$$
By the Molien's theorem (1897)
$$Mol(0,T)=\frac{1}{|G|}\sum_{g\in G}\frac1{\det(1-T g)}\,.$$
An easy generalization (see Appendix \ref{extmolien}) of the Molien's theorem provides the formula
\begin{equation}\label{molien}Mol(v,T)
=\frac{1}{|G|}\sum_{g\in G}\;\frac{\det(1+vg)}{\det(1-Tg)}
\,.
\end{equation}
The goal of this paper is to show a relation between the local equivariant Hirzebruch class of the quotient singularity $\C^n/G$ and the Molien series of (the chosen representation of)~$G$. We prove
\begin{theorem}\label{th1} For any quotient singularity we have
$$td_y^{\C^*}(X\hookrightarrow M)=eu(M,0)Mol(y T,T)\,,$$
i.e.
$$H(y,T)=Mol(yT,T)\,.$$
\end{theorem}
This kind of interpretation of the local Hirzebruch class (or rather the Todd class for  $y=0$)  has appeared already in \cite[(3.8)]{Ba} in a slightly different context.

The equality can be understood as a form of the McKay correspondence: a relation between geometry of the resolution of the quotient singularity and algebraic properties of the action.
We prove Theorem \ref{th1} by applying Lefschetz-Riemann-Roch theorem. A similar LRR-type formula  for elliptic genus  was proved by Borisov and Libgober \cite{BoLi}. Their formula applies to crepant resolutions of global quotients of projective manifolds. Our approach is local.
By the result of Borisov and Libgober specialized to the Hirzebruch class of a crepant resolution of $f:\widetilde X\to\C^n/G$ we obtain that
$$\frac{f_*td_y^\T(\widetilde X)}{eu(M,0)}=\sum_{h\in Conj(G)}(-y)^{age(h)}Mol(C(h),(\C^n)^h;yT,T)\,,$$
where $Mol(C(h),(\C^n)^h;yT,T)$ is the Molien series of the representation of the centralizer $C(h)$ on the space of fixed points $(\C^n)^h$. By $age(h)$ we understand $\sum_{k=1}^n \lambda_k$, where $e^{2\pi i\lambda_k}$, $k=1,\dots,n$ are the eigenvalues of $h$ and $\lambda_k\in[0,1)\cup \Q$.

\begin{remark}\rm The elliptic genus is more general than the Hirzebruch $\chi_y$--genus. Elliptic characteristic class specializes to $td_y$ by a limit process. It might be interesting to see what is an interpretation of the local elliptic class from the representation theory point of view.
\end{remark}

\section{Lefschetz-Riemann-Roch}\label{secLRR}
First let us recall  results of \cite{BFQ} which  lead to a formula for the Todd class of the quotient variety $td_{y=0}(Y/G)$. After that we review  \cite{CMSS}, adopting the notation to our purposes. These strengthened version of Lef\-schetz-Riemann-Roch allows to compute full Hirzebruch class. The proof of Theorem \ref{th1} is just checking that the methods of \cite{BFQ} and \cite{CMSS} apply in the equivariant case and interpreting the result for $Y=\C^n$.

\subsection{LRR for the Todd class}
Suppose $Y$ is a smooth quasiprojective variety on which a finite group $G$ acts.
Set  $$Y^g=\{x\in Y\;|\;gx=x\}\,.$$
Let $V\subset Y^g$ be a connected component.
We will define a certain element $\lambda_V^g\in K(V)\otimes \C$ in the $K$-theory of coherent sheaves.
 Let $N_{Y/V}^*$ be the conormal bundle
\begin{equation}\lambda^g_V =\sum_{a\text{ root of unity}} \sum_{k=0}^{\codim(V)}(-1)^k a \cdot(\Lambda^k N_{Y/V}^*)_{g,a}\in K(V)\otimes \C\end{equation}
where $(\Lambda^k N_{Y/V}^*)_{g,a}\in K(V)$
is the subbundle of $\Lambda^k N_{Y/V}^*$ on which $g$ acts with eigenvalue~$a$.
If $ N_{Y/V}=\bigoplus_{k=1}^{\codim(V)} L_k$ was a direct sum  of line bundles with $g$ acting on $L_k$ via the multiplication by $a_k(g)$ then we would have $$\lambda^g_V=\prod_{k=1}^{\codim(V)}(1-a_k(g)^{-1} [ L_k])\,.$$
By \cite[Lemma 4.3]{Don} this element is invertible. Let $$L^{(g)}Y=\sum_{V\text{ component of } Y^g} (\lambda_V^g)^{-1}\in K(Y^g)\otimes \C\,.$$
Denote by $X=Y/G$  the quotient variety.
 By \cite[\S4 (2)]{BFQ}    we have an equality \begin{equation}\label{K-LRR}[\OOO_X]=\frac1{|G|}\sum_{g\in G}\pi^g_*(L^{(g)}Y)\in K(X)\otimes \C\,,\end{equation}
where $\pi^g:Y^g\to Y/G=X$ is the projection of the fixed point set to the quotient.
The  Todd class is obtained by applying the Grothendieck-Riemann-Roch transformation
$$\begin{matrix}&K(M)&\to &H^*(M)\\
&{\mathcal F}&\mapsto& td(M)ch({\mathcal F})\,,\end{matrix}$$
to formula (\ref{K-LRR}). Here $td(Y)$ is the classical Todd class  and $ch(-)$ is the Chern character. We obtain
\begin{align*}td(M)ch(\OOO_X)=&td(M)ch\left(\frac1{|G|}\sum_{g\in G}\pi^g_*(L^{(g)}Y)\right)=\\\stackrel{GRR}=&\frac1{|G|}\sum_{g\in G}\pi^g_*(td(Y^g)ch(L^{(g)}Y))\,.\end{align*}
This is an expression for the image of Baum-Fulton-MacPherson class in $H^*(M)$, which coincides for rational singularities with $td_{y=0}(X\hookrightarrow M)$ (see \cite[Example 3.2]{BSY} or \cite[\S 14]{We3}).
 Each morphism $\pi^{g}$ can be factorized as $\pi\circ \iota^g$, where $\iota^g:Y^g\to Y$ is the inclusion of the fixed point set. We write
$$td_{y=0}(X\hookrightarrow M)=\frac1{|G|}\pi_*\sum_{g\in G}\iota^g_*(td(Y^g)ch(L^{(g)}Y))\,.$$
The class $ch(L^{(g)}Y)$ is the sum of contributions $ch((\lambda^g_V)^{-1}))$ coming from various components of $Y^g$ and the  expression for $\iota^g_*(td(Y^g)ch((\lambda^g_V)^{-1}))$  in terms of the Chern roots is the following:
\begin{multline}\label{LformulaTd}\iota^g_*\left(td(Y^g)\cdot
\prod_{k=1}^{\codim(V)}\frac1{1-a_k(g)^{-1} e^{-x_k}}\right)
=\\=\iota^g_*\left(\prod_{\ell=\codim(V)+1}^{\dim( Y)}\frac1{1- e^{-x_\ell}}
\cdot\prod_{k=1}^{\codim(V)}\frac1{1-a_k(g)^{-1} e^{-x_k}}\right)
=\\=
\prod_{k=1}^{\codim(V)}x_k\cdot\prod_{\ell=1+\codim(V)}^{\dim( Y)}\frac{x_\ell}{1- e^{-x_\ell}}
\cdot\prod_{k=1}^{\codim(V)}\frac{1}{1-a_k(g)^{-1} e^{-x_k}}\,.\end{multline}
where \begin{itemize}
\item $x_k$ for $k=1,\dots, \codim(V)$ are the roots of $(N_{Y/V})$,
\item the eigenvalue corresponding to $x_k$ is $a_k(g)$,
\item $x_\ell$ for $\ell=\codim(V),\dots, \dim(Y)$ are the roots of $T_V$.
\end{itemize} Finally we can write
\begin{equation}\label{Lformula}\iota^g_*(td(Y^g)ch((\lambda^g_V)^{-1}))=
\prod_{k=1}^{\dim(Y)}\frac{x_k}{1-a_k(g)^{-1} e^{-x_k}}
\end{equation}
setting $a_k(g)=1$ for $k>\codim(V)$.

\subsection{LRR for the Hirzebruch class}
The same argument can be carried on for the full Hirzebruch class, as it is done in \cite[Theorem 5.1]{CMSS}\footnote{The authors of \cite{CMSS} use the notation $\widetilde T_y$ for ours $td_y$.}. For simplicity let us assume that $Y^g$ has one component. We
obtain
\begin{equation}\label{hirquo}
td_y(X\hookrightarrow M)=\frac1{|G|}
\sum_{g\in G}\pi^g_*\left(td_y(Y^g)\prod_{\theta\in(0,2\pi)}
\TU_y^\theta\left((N_{Y/Y^g})_{g,e^{i\theta}}\right)\right)
\,,\end{equation}
where for $\theta\in[0,1)$ the entry $\TU_y^\theta\left((N_{Y/Y^g})_{g,e^{i\theta}}\right)\in H^*(Y^g)\otimes \C[y]$ is expressed in terms of $z_s$, the Chern roots of $(N_{Y/Y^g})_{g,e^{i\theta}}$ as follows:
\begin{align*}\TU_y^\theta\left((N_{Y/Y^g})_{g,e^{i\theta}}\right)&=\prod_{s=1}^{\dim((N_{Y/Y^g})_{g,e^{ i\theta} })}\frac{1+y\, e^{-z_s-i\theta}}{1- e^{-z_s-i\theta}}=\\&=\prod_{k:\;e^{i\theta}=a_k(g)}\frac{1+y\, a_k(g)^{-1}e^{-x_k}}{1- a_k(g)^{-1}e^{-x_k}}\,,\end{align*}
see~\cite[\S2.1(vi) and Definition 2.2]{CMSS}.
The formula for $td_y$ is almost the same as for $\TU_y^\theta$ with $\theta=0$:
$$td_y(Y^g)=\prod_{\ell=\codim(V)+1}^{\dim( Y)}\frac{x_\ell(1+y\, e^{-x_\ell})}{1- e^{-x_\ell}}$$
Therefore
\begin{multline}\label{Lformula0}\iota^g_*\left(td_y(Y^g)\prod_{\theta\in(0,2\pi)}
\TU_y^\theta\left((N_{Y/Y^g})_{g,e^{i\theta}}\right)\right)=\\
=\iota^g_*\left(\prod_{\ell=\codim(V)+1}^{\dim( Y)}\frac{x_\ell(1+y\, e^{-x_\ell})}{1- e^{-x_\ell}}
\cdot\prod_{k=1}^{\codim(V)}\frac{1+y\,a_k(g)^{-1} e^{-x_k}}{1-a_k(g)^{-1} e^{-x_k}}\right)
=\\
=\prod_{k=1}^{\codim(V)}x_k\cdot\prod_{\ell=1+\codim(V)}^{\dim( Y)}\frac{x_\ell(1+y\, e^{-x_\ell})}{1- e^{-x_\ell}}
\cdot\prod_{k=1}^{\codim(V)}\frac{1+y\,a_k(g)^{-1} e^{-x_k}}{1-a_k(g)^{-1} e^{-x_k}}\,.\end{multline}
 As before, we write that class in a closed formula
\begin{equation}\iota^g_*\left(td_y(Y^g)\prod_{\theta\in(0,2\pi)}
\TU_y^\theta\left((N_{Y/Y^g})_{g,e^{i\theta}}\right)\right)=\prod_{k=1}^{\dim(Y)}\frac{x_k(1+y\,a_k(g)^{-1} e^{-x_k})}{1-a_k(g)^{-1} e^{-x_k}}
\end{equation}
setting $a_k(g)=1$ for $k>\codim(V)$.
We remark that the proof given in \cite{CMSS} uses $K$-theory of mixed Hodge modules of M. Saito \cite{Sai}.
\medskip

\subsection{Proof of Theorem \ref{th1}.}
Suppose that the group $\T$ acts on $Y$ and the action of $\T$ commutes with the action of $G$. Then $X=Y/G$ admits an action of $\T$ such that the projection
$$\pi:Y\to X$$ is  $\C^*$-equivariant.
We claim that formula (\ref{hirquo}) holds in equivariant cohomology. To justify that let us recall the definition of the equivariant Hirzebruch class  via approximation \cite[Def. 7.1]{We3}:
$$td_y^\T(X)=\lim_{m\to \infty} p_m^*\left(td_y(B_m)^{-1}\right)\cap td_y(E_m\times^\T X)\,,$$
where $E_m\to B_m$ is an approximation of the universal $\T$-bundle with $B_m$ being a smooth algebraic variety\footnote{The standard model for $B_m$ is $\P^m$ and $E_m=\C^{m+1}\setminus\{0\}$.} and
$p_m:E_m\times^\T Y\to B_n$ the associated approximation of the Borel construction.   The group $G$ acts on $E_m\times^\T Y$ and by the functoriality of the Riemann-Roch transformation  we obtain the generalization of formula (\ref{hirquo}) for equivariant homology
\begin{equation}\label{hirquoeq}td^\T_y(X)=\frac1{|G|}
\sum_{g\in G}\pi _*\iota^g_*\left(td_y^\T(Y^g)\prod_{\theta\in(0,2\pi)}
\TU_y^\theta\left((N_{Y/Y^g})_{g,e^{i\theta}}\right)\right)\in \hat H_{\T,*}(X)\otimes \C[y]\,.\end{equation}
Here the equivariant class $\iota^g_*\left(td_y^\T(Y^g)\prod_{\theta\in(0,2\pi)}
\TU_y^\theta\left((N_{Y/Y^g})_{g,e^{i\theta}}\right)\right)$ is given by the same formula (\ref{Lformula}), which is valid in the equivariant cohomology of $Y^g$.
Since $X$ has only quotient singularities, we have Poincar\'e duality $$\hat H_{T,*}(X)\simeq \hat H^{2\dim( X)-*}_\T(X)\,.$$

Let $G\subset GL_n(\C)$ be a finite group and let $Y=\C^n$ be the natural representation of $G$. Assume $Y^G=\{0\}$. The scalar action of $\T$ commutes with the action of $G$. We apply formula
(\ref{hirquoeq}) to computing the equivariant Hirzebruch class.
As before, each morphism $\pi^{g}$ is factorized as $\pi\circ \iota^g$, where $\iota^g:(\C^n)^g\to \C^n$ is the inclusion.
Let us compute the equivariant version of the class computed in (\ref{Lformula}):
\begin{multline}
\iota^g_*\left(td_y^\T(Y^g)\prod_{\theta\in(0,2\pi)}
\TU_y^\theta\left((N_{Y/Y^g})_{g,e^{i\theta}}\right)\right)=\\
=t^{\codim((\C^n)^g)}\prod_{\ell=1}^{\dim( (\C^n)^g)}\frac{t(1+y\, e^{-t})}{1- e^{-t}}
\cdot\prod_{k=1}^{\codim((\C^n)^g)}\frac{(1+y\,a_k(g)^{-1} e^{-t})}{1-a_k(g)^{-1} e^{-t}}\,.
\end{multline}
The numbers $a_k(g)\in\C$ for $k=1,\dots ,\codim ((\C^n)^g)$ are the eigenvalues of $g$ which are different from 1. Finally, we can write
\begin{equation}\iota^g_*\left(td_y^\T(Y^g)\prod_{\theta\in(0,2\pi)}
\TU_y^\theta\left((N_{Y/Y^g})_{g,e^{i\theta}}\right)\right)=t^n\prod_{k=1}^{n}\frac{1+y\,a_k(g)^{-1} T}{1-a_k(g)^{-1} T}\,,\end{equation}
where $a_k(g)$ for $k=1,\dots ,n$ are all the eigenvalues of $g$ acting on $\C^n$ and $T=e^{-t}$.
We assume that the ambient space $M$ containing $X$ is another vector space. Let $j:X\to M$ be the inclusion. For any element $\alpha\in H^*_\T(Y)$ we have
$$\frac{(j_*\pi_*\alpha)_{|0}}{eu(M,0)}=\frac{\alpha_{|0_Y}}{eu(Y,0)}\in \hat H^*_{\C^*}(pt)\otimes\C[y,t^{-1}]\simeq \C[[t]]\otimes\C[y,t^{-1}]\,.$$
Setting $\alpha=\iota^g_*\left(td_y^\T(Y^g)\prod_{\theta\in(0,2\pi)}
\TU_y^\theta\left((N_{Y/Y^g})_{g,e^{i\theta}}\right)\right)$ and summing over the elements of $g$ we obtain the expression for $\frac{td_y^{\C^*}(X\hookrightarrow M)_{|0}}{eu(M,0)}$.
We change the order of the summation replacing $g$ by $g^{-1}$ and we conclude that
\begin{align*}\frac{td_y^{\C^*}(X\hookrightarrow M)_{|0}}{eu(M,0)}&=\frac{j_*td^{\C^*}_y(X)_{|0}}{eu(M,0)}=\\ &=
\sum_{g\in G}\prod_{k=1}^{n}\frac{1+a_k(g)yT}{1-a_k(g) T}
=\sum_{g\in G}\frac{\det(1+yTg)}{\det(1-Tg)}
\end{align*}
By formula (\ref{molien}) this is exactly the expression for the extended Molien series with $v=yT$. Hence
$$td_y^{\C^*}(X\hookrightarrow M)_{|0}=eu(M,0)Mol(vT,T)\,.$$
The restriction to $0$ is an isomorphism on equivariant cohomology since $M$ is contractible. Thus we obtain the claim. \qed

\section{The functional equation}\label{section_functional_eq}

The function $$H(y,T)= \frac1{eu(M,0)}td_y^\T(\C^n/G\hookrightarrow M)$$ has some symmetries.

\subsection{General linear group}
The basic symmetry holds for arbitrary $G\subset GL_n(\C)$, further symmetries appear for subgroups of $SL_n(\C)$ or $Sp_n(\C)$.
\begin{proposition}[Duality]\label{prop_formula_Poincare}
For any $n$-dimensional quotient singularity we have
$$H ( 1/y, 1/T)(-y)^n=H(y,T)$$
\end{proposition}

\begin{proof}
\begin{multline*}
H(1/y, 1/T) = Mol(1/(yT), 1/T) = \frac{1}{|G|}\sum_{g\in G} \prod_{k=1}^n
\frac{1+a_k(g)(yT)^{-1}}{1-a_k(g)T^{-1}} = \\ =\frac{1}{|G|}\sum_{g\in G} \prod_{k=1}^n \frac{a_k(g)(yT)^{-1}(yTa_k(g)^{-1} + 1)}{a_k(g)T^{-1}(Ta_k(g)^{-1}-1)} = \frac{1}{|G|(-y)^n} \sum_{g\in G} \prod_{k=1}^n \frac{1+ yTa_k(g)^{-1}}{1-Ta_k(g)^{-1}}\,.
\end{multline*}
We replace $g$ by $g^{-1}$ in the summation and we note that the eigenvalues of $g^{-1}$ are inverses of the eigenvalues of $g$:
\begin{equation*}
H(1/y, 1/T)= \frac{1}{|G|(-y)^n} \sum_{g\in G} \prod_{k=1}^n \frac{1+ yTa_k(g^{-1})}{1-Ta_k(g^{-1})} = (-y)^{-n} Mol(yT,T) = (-y)^{-n} H(y,T)\,.
\end{equation*}
\end{proof}
This duality does not hold for arbitrary singularities. For example, for the affine cone over a curve of degree 4 in $\P^2$ we have
$$H(y,T)=2(1+y)\frac{
T +  T^2 +  (3 T- T^2) y}{(1-T)^2}+1$$
and the duality does not hold. (The formula for the Hirzebruch classes of affine cones is given in \cite[Prop.~10.3]{We3}.)

\begin{proposition}[Divisibility]\label{div}
For any finite subgroup $G\subset GL_n(\C)$
the polynomial
$H(y,T)-1\in \Q(T)[y]$ is divisible by $y+1$.
\end{proposition}

\begin{proof}
We have
\begin{equation*}
H(y,T)-1 = Mol(yT,T)-1 = \frac{1}{|G|}\sum_{g\in G} \frac{\det(1+yTg)-\det(1-Tg)}{\det(1-Tg)}.
\end{equation*} The expression vanishes for $y=-1$, so divisibility in $\Q(T)[y]$ follows.
\end{proof}
Divisibility can be explained geometrically by the fact that
$$eu(M,0)(H(y,T)-1)=td_y^\T((\C^n\setminus \{0\})/G\hookrightarrow M)$$
and $(\C^n\setminus \{0\})/G$ is a sum of nonconstant orbits of $\T$. Each orbit is isomorphic to $\C^*$ and $\chi_y(\C^*)=-(y+1)$. The divisibility follows from the multiplicative properties of $\chi_y$-genus and the Hirzebruch class.

\subsection{Special linear group}
\begin{proposition}[$SL$-duality]
For any finite subgroup $G\subset SL_n(\C)$
$$H(y,T)=\frac{H( y T^2, 1/T)}{(-T)^n}
$$
\end{proposition}

\begin{proof}
First note that  for $G\subset SL_n(\C)$ the sequence of exterior powers is symmetric, i.e. $\Lambda^l(\C^n) \simeq \Lambda^{n-l}(\C^n)$. Hence we have
$$Mol(v,T) = Mol(1/v,T)v^n.$$
Then, applying  Prop.~\ref{prop_formula_Poincare}, we obtain

\begin{multline*}
\frac{H( y T^2, 1/T)}{(-T)^n} = \frac{Mol(yT,1/T)}{(-T)^n} = \frac{Mol(1/(yT),1/T)}{(-T)^n}\cdot y^nT^n = \\ = H(1/y,1/T)(-y)^n = \frac{H(y,T)}{(-y)^n}\cdot (-y)^n = H(y,T).
\end{multline*}
\end{proof}

The $SL$-duality means that the coefficients of $H(y,T)$ as a polynomial in $y$ are palindromic with respect to $T$. This kind of duality
does not hold in general.
For example for the quotient of $\C^2$ by $\Z_n$ acting diagonally we have
$$H(y,T)=\frac{1+(n-1)T^n+2n T^n y+((n-1)T^n+T^{2n})y^2}{(1-T^n)^2}\,.$$

\subsection{Symplectic group}
\begin{proposition}[Symplectic divisibility]\label{sdiv}
If $G\subset Sp_{n}(\C)\subset SL_{2n}(\C)$ is a finite symplectic group then $H(y,T)-(-y)^n$ is divisible by $yT^2+1$.
\end{proposition}

\begin{proof}
We have
\begin{multline*}
H(y,T)-(-y)^n = Mol(yT,T)-(-y)^n =\\= \frac{1}{|G|}\sum_{g\in G} \frac{\det(1+yTg)-(-y)^n\det(1-Tg)}{\det(1-Tg)},
\end{multline*}
so it suffices to prove that for any $g \in G$ the polynomial $$P(y,T) = \det(1+yTg)-(-y)^n \det(1-Tg)$$ is divisible by $(yT^2+1)$. If we view it as a polynomial in $y$ over the field of rational functions $\Q(T)$, we need only to show that  $P(-\frac{1}{T^2},T) = 0$. One looks at
$$P\left(-\frac{1}{T^2},T\right) = \det\left(1-g\frac{1}{T}\right)-\frac{1}{T^{2n}}\det(1-Tg)$$
which is 0 if and only if the coefficients $d_0,\ldots,d_{2n}$ of the polynomial $\det(1-Tg)$ form a symmetric sequence: $d_k = d_{2n-k}$ for any $k \in \{0,\ldots,n\}$. Equivalently, the sequence of eigenvalues $(a_1(g),\ldots,a_{2n}(g))$ of $g$ is a permutation of $(a_1(g)^{-1},\ldots,a_{2n}(g)^{-1})$. And symplectic matrix groups have this property.
\end{proof}

\begin{corollary}\label{symp-postac} The polynomial $H(y,T)-(-y)^n$ for a symplectic quotient $\C^{2n}/G$ is divisible by $(y+1)(yT^2+1)$.
\end{corollary}

\begin{proof} Observe that $(-y)^n-1$ is divisible by $y+1$ and use Propositions \ref{div} and \ref{sdiv}.\end{proof}

 \begin{corollary}\label{postac} For a surface quotient singularity the polynomial $H(y,T)$ is determined by $H(0,T)$. It is of the form
$$(y+1)(f(T)+yf(1/T))+1\,$$
where $f(T)=H(0,T)-1$.
For symplectic quotients the Hirzebruch class is equal to
$$H(y,T)=(y+1)(yT^2+1)H(0,T)-y\,.$$\end{corollary}

\begin{proof}The first statement follows from Propositions \ref{prop_formula_Poincare} and \ref{div} since for a surface singularity $H(y,T)$ is of degree two as a polynomial in $y$. The second statement follows from Corollary~\ref{symp-postac}.\end{proof}

\section{Hirzebruch class of quotient surfaces}
\label{sectionduval}
\subsection{Du Val singularities as hypersurfaces}
\label{duval}
The Hirzebruch classes of Du Val singularities, i.e. the symplectic quotients of $\C^2$, are surprisingly simple.
All of these quotients
can be realized as hypersurfaces  $X\subset\C^3$ given by quasihomogeneous polynomials. In general, for a hypersurface in a smooth variety $X\subset M$ the image of the Baum-Fulton-MacPherson class in $H^*(M)$  is given by the formula \begin{equation}\label{BFM}td(M)ch(\OOO_X)=td(M)(1-e^{-[X]}),\end{equation}
 see \cite[18.3.5]{Fu2}.  The embedding into $\C^3$ can be made equivariant and the  formula (\ref{BFM}) holds for equivariant classes. Since the Baum-Fulton-MacPherson class coincides with the class $td^\T_{y=0}$, we obtain
$$ H(0,T)=\frac{td^\T(\C^3)}{eu(\C^3,0)}ch^\T(\OOO_X)=\frac{1-T^d}{(1-T^{w_1})(1-T^{w_2})(1-T^{w_3})}\,,$$
where $w_i$ are the weights of the action of $\T$ on $\C^3$ and $d$ is the weighted degree of the polynomial defining the hypersurface.
By Corollary \ref{postac} we have
$$ H(y,T)=(y+1)(yT^2+1)\frac{1-T^d}{(1-T^{w_1})(1-T^{w_2})(1-T^{w_3})}-y\,.$$
\subsection{Hirzebruch classes of Du Val singularities}\label{duvallist} We list below the Hirzebruch classes for the series $A$, $D$ and $E$:\medskip

$\bullet$ Group $\Z_{n}$,  singularity $A_{n-1}$, $x^n+y^2+z^2$

$$(y+1)\left(y
   T^2+1\right)\frac{1-T^{2 n}}{\left(1-T^2\right)
   \left(1-T^{n}\right)^2}-y$$
By Theorem \ref{th1} this expression is equal to
$$\frac 1n\sum_{k=0}^{n-1}\frac{1+2\cos(\frac{2k\pi }n) y\, T+y^2T^2}{1-2\cos(\frac{2k\pi }n) T+T^2}\,.$$

$\bullet$ Binary dihedral group $BD_{4(n-2)}$, singularity $D_{n}$, $x^{n-1}+y^2x+z^2$
$$(y+1) \left(y
   T^2+1\right)\frac{1-T^{4 n-4}}{\left(1-T^4\right)
   \left(1-T^{2 (n-2)}\right) \left(1-T^{2
   (n-1)}\right)}-y$$

$\bullet$ Binary tetrahedral group $BT_{24}$, singularity $E_6$, $x^4+y^3+z^2$
$$
(y+1)  \left(y T^2+1\right)\frac{1-T^{24}}{\left(1-T^6\right)
   \left(1-T^8\right)
   \left(1-T^{12}\right)}-y$$

$\bullet$ Binary octahedral group $BO_{48}$, singularity $E_7$, $x^3+xy^3+z^2$
$$(y+1)  \left(y T^2+1\right)\frac{1-T^{36}}{\left(1-T^{12}\right)\left(1-T^8\right)
   \left(1-T^{18}\right)}-y$$

$\bullet$ Binary icosahedral group $BI_{60}$, singularity $E_8$, $x^5+y^3+z^2$
$$(y+1) \left(y T^2+1\right)\frac{1-T^{60}}{\left(1-T^{12}\right)
   \left(1-T^{20}\right)
   \left(1-T^{30}\right)}-y$$

The Chern-Schwartz-MacPherson classes are specializations of Hirzebruch classes: $c^{SM}=\lim_{y\to -1 }H(y,e^{-(1+y)t})$. The local equivariant version of Chern-Schwartz-MacPherson classes is studied in
\cite{We1}. Here are the formulas for Du Val singularities.
{\Large$$\begin{matrix} ^{A_{n-1}}& ^{D_{n}}& ^{E_6}&^{E_7}& ^{E_8}\\
\frac{n t^2+2 t+1}{n t^2},&\frac{4 (n-2) t^2+2 t+1}{4 (n-2) t^2},&\frac{24 t^2+2 t+1}{24 t^2},&\frac{48
   t^2+2 t+1}{48 t^2},&\frac{120 t^2+2 t+1}{120 t^2}.\end{matrix}$$}

\subsection{Hirzebruch class of surface singularities via resolution}
Suppose, that $(S,0)\subset(M,0)$ is a germ of isolated surface singularity embedded in a smooth variety. Suppose a torus $\T$ acts on $M$ preserving $S$ and $0$. As before, $M$ can be taken as a vector space with a linear action of~$\T$. Let $\widetilde S\to S\subset M$ be an equivariant resolution of $S$ with the exceptional divisor having simple normal crossings. By functoriality and additivity of the Hirzebruch class we have
\begin{align*}td^\T_y(S\hookrightarrow M)=&p_*td^\T_y(\widetilde S)-p_*td^\T_y(E\hookrightarrow \widetilde S)+td^\T_y(0\hookrightarrow M)\\
=&p_*td^\T_y(\widetilde S)+(1-\chi_y(E))[0]\,,\end{align*}
where $p:\widetilde S\to M$ is the resolution map composed with the embedding into $M$,~ $E=\bigcup_{i=1}^k E_i$ is the exceptional divisor and $[0]\in H^4_\T(M)$ is the class of the point. The $\chi_y$-genus of $E$ can be computed by additivity:
$$\chi_y(E)=\sum_{i=1}^k\chi_y(E_i)-\ell\,\chi_y(pt)=\sum_{i=1}^k (1-g_i)(1-y)-\ell$$
where $g_i$ is the genus of $E_i$ and $\ell$ is the number of intersection points.
If $E$ is a tree of rational curves then
$$\chi_y(E)=-k(1-y)-(k-1)=-k\,y+1$$
and
$$td^\T_y(S\hookrightarrow M)=p_*td^\T_y(\widetilde S)+k\, y[0]\,.$$
To compute the push forward $p_*td^\T_y(\widetilde S)$ one can apply Atiyah-Bott or Berline-Vergne localization, \cite{AtBo2,BeVe}, which holds in the relative case by \cite[Corollary 3.2]{PeTu}.
 %[Formula 3.8]
 If the action of $\T$ has only isolated fixed points, then
\begin{align*}p_*td^\T_y(\widetilde S)_{|0}&=eu(M,0)\sum_{p\in \widetilde S^\T}\frac1{eu(\widetilde S,p)}td^\T_y(\widetilde S)_{|p}=\\
&=eu(M,0)\sum_{p\in \widetilde S^\T} \frac{1+y T^{w_1(p)}}{1- T^{w_1(p)}}\frac{1+y T^{w_2(p)}}{1- T^{w_2(p)}}\,,\end{align*}
where $w_i(p)$ for $i=1,2$ are the weights of the $\T$ action on the tangent space $T_p\widetilde S$.
If the fixed point set $\widetilde S^\T$ is not finite then the expression for the Hirzebruch class has an additional summand corresponding to each fixed component $E_{fix}$
\begin{equation}\label{wyliczenie}\int_{E_{fix}} \frac1{c_1(N_{fix})}td_y^\T(\widetilde S)_{|E_{fix}}
=\int_{E_{fix}} td_y^\T(E_{fix})
\frac
{1+y e^{-c_1(N_{fix})}}
{1-  e^{-c_1(N_{fix})}}\,,\end{equation}
where $N_{fix}$ is the normal bundle to the fixed component.

We will illustrate the computations by the example of Du Val singularities.
Among Du Val singularities only $A_n$ with $n$ even has isolated fixed points in the resolution. For the remaining Du Val singularities there always exists exactly one fixed component:

-- the central component of $E$ for the series $A_n$ with $n$ odd,

-- the component which meets three other components for the series $D_n$ and $E_6$, $E_7$, $E_8$.
\smallskip

\noindent To describe the situation we encode the weights in the Dynkin diagram: the edges, i.e. the intersections of divisors,  are labelled by the weights of the action of the torus on the tangent space at the intersection point.  The loose edges of the diagram correspond to the fixed points which are not the intersection points.
Let us give a few examples:\smallskip

$\bullet$ The singularity $A_6$
\[
\begin{diagram}
\node{} \arrow{e,t,-}{7,-5}
 \node{\bullet} \arrow{e,t,-}{5,-3}
 \node{\bullet} \arrow{e,t,-}{3,-1}
 \node{\bullet}\arrow{e,t,-}{1,1}
 \node{\bullet} \arrow{e,t,-}{-1,3}
 \node{\bullet} \arrow{e,t,-}{-3,5}
 \node{\bullet} \arrow{e,t,-}{-5,7}
 \node{}
\end{diagram}
\]

$\bullet$ The singularity $A_5$
\[
\begin{diagram}
 \node{} \arrow{e,t,-}{6,-4}
 \node{\bullet} \arrow{e,t,-}{4,-2}
 \node{\bullet} \arrow{e,t,-}{2,0}
 \node{E_{fix}} \arrow{e,t,-}{0,2}
 \node{\bullet} \arrow{e,t,-}{-2,4}
 \node{\bullet} \arrow{e,t,-}{-4,6}
 \node{}
 \end{diagram}
\]

$\bullet$ The singularity $D_5$
\[
\begin{diagram}
\node{} \node{}\node{}\node{}\node{}\node{}  \\
 \node{}\node{}\node{}\node{}\node{\bullet}\arrow{ne,t,-}{-2,4}  \\
 \node{}\arrow{e,t,-}{6,-4}
 \node{\bullet} \arrow{e,t,-}{4,-2}
\node{\bullet} \arrow{e,t,-}{2,0}
 \node{E_{fix}}\arrow{ne,t,-}{0,2}\arrow{se,t,-}{0,2}\\
 \node{}\node{} \node{}\node{}\node{\bullet}\arrow{se,t,-}{-2,4}\\
 \node{} \node{}\node{}\node{}\node{}\node{}
\end{diagram}
\]

$\bullet$ The singularity $E_7$
\[
\begin{diagram}
 \node{} \node{} \node{} \node{} \arrow{s,r,-}{\begin{matrix}\hfill_4\\^{-2}\end{matrix}}\\
 \node{} \node{} \node{} \node{\bullet} \arrow{s,r,-}{\begin{matrix}_2\\^0\end{matrix}}\\
  \node{}\arrow{e,t,-}{6,-4}
 \node{\bullet} \arrow{e,t,-}{4,-2}
\node{\bullet} \arrow{e,t,-}{2,0}
 \node{E_{fix}}\arrow{e,t,-}{0,2}
 \node{\bullet} \arrow{e,t,-}{-2,4}
 \node{\bullet} \arrow{e,t,-}{-4,6}
 \node{\bullet}\arrow{e,t,-}{-6,8}
 \node{}
\end{diagram}
\]

The weights are computed in the following way:
$A_n$ admits an action of two-dimensional torus, so it is an affine toric surface. The structure of the resolution can be read from the fan. The singularity $D_n$ is a quotient of $A_{2n}$ by $\Z_2$ and the series $E_k$ can be analyzed directly: the curve with three intersection points has to be fixed by $\T$ and the action on the remaining curves can be computed inductively: the action on the normal direction determines the self-intersection which is $-2$.

  The neighbourhood of the fixed component for Du Val singularities is equivariantly isomorphic to the resolution of $A_1$ singularity. The contribution of that component is equal to
$$\frac{td_y^\T(\C^2/\Z_2\hookrightarrow M)}{eu(M,0)}-y=
\frac{(y+1)\left(y
   T^2+1\right) \left(T^2+1\right)
}{(1-T^2)^2}-2y$$
by \S\ref{duval} or by a direct computation.
For example, from the formula (\ref{wyliczenie}) and the diagram above we compute the Hirzebruch class for $E_7$:
\begin{multline*}\frac1{eu(M,0)}p_*td^\T_y(\widetilde S)=\frac{(y+1)\left(y
   T^2+1\right) \left(T^2+1\right)
}{(1-T^2)^2}-2y+3\frac
{(1 + yT^{-2}) (1 + T^4 y)}{(1 - T^{-2}) (1 - T^4)} +
 \\+2 \frac{(1 + yT^{-4}) (1 + T^6 y)}{(1 - T^{-4}) (1 - T^6)} + \frac{(1 + y
    T^{-6}) (1 + T^8 y)}{(1 - T^{-6}) (1 - T^8)}\,.\end{multline*}
After simplification we obtain
$$\frac{(y+1)  \left(y T^2+1\right)(1-T^{36})}{\left(1-T^{12}\right)\left(1-T^8\right)
   \left(1-T^{18}\right)}-8y$$
To get the formula for $\frac1{eu(M,0)}td^\T_y( S\hookrightarrow M)$ one has to add $7y$.

\subsection{Relation with Poincar\'e series}

In \cite{CDG} there are constructed Poincar\'e series of surface singularities. They are generating series for multifiltrations in $\OOO_X$ defined by valuations in the components of the exceptional divisors of the minimal resolution. This filtration is related to the grading defined by the torus action. For the singularities $A_{2m-1}$, $D_n$, $E_6$, $E_7$ and  $E_8$ there is a component which is fixed by the torus. When we specialize the Poincar\'e series to that component we obtain the classical Molien series ($v=0$) which coincides with $H(0,T)$.
\begin{itemize}
\item General form of Poincar\'e series for
$A_n$ singularity (notation from \cite[Ex.~1]{CDG})
$$\frac{1-\left(\prod _{k=1}^n
   t_k\right){}^{n+1}}{\left(1-\prod
   _{k=1}^n t_k\right) \left(1-\prod
   _{k=1}^n t_k^k\right) \left(1-\prod
   _{k=1}^n t_k^{-k+n+1}\right)}$$

\item For $n=2m-1$ substituting
$t_m=T^2$ and $t_k=1$ for $k\not=m$ we obtain the function $H(0,T)=Mol(0,T)$.

\item For the singularity $D_4$ the Poincar\'e series (\cite[Ex.~2]{CDG}) is
$$\frac{\left(1-t_1 t_2 t_3 t_4^2\right)
   \left(t_1^2 t_2^2 t_3^2
   t_4^3+1\right)}{\left(1-t_1^2 t_2
   t_3 t_4^2\right) \left(1-t_1 t_2^2
   t_3 t_4^2\right) \left(1-t_1 t_2
   t_3^2 t_4^2\right)}$$
 When we substitute
 $t_4= T^2$ and $t_i=1$ for $i\not=4$
we obtain the function $H(0,T)=
\frac{T^6+1}{\left(1-T^4\right)^2}$.

\item  For $D_5$ after the substitution
$t_3 = T^2$ and $t_i= 1$ for $i\not=3$
we obtain the function $H(0,T)=\frac{T^8+1}{\left(1-T^4\right)
   \left(1-T^6\right)}$.

\item  But there are more possible  substitutions; for example for
$A_n$: $t_n = t_1 = T$ and $t_k=1$ for $k\not=1,\,n$,
which works also for any $n$.

\item For $n=2m$: $t_m=t_{m+1}=T$, and $t_k=1$ for $k\not=1$.
\end{itemize}

The relation between the equivariant Hirzebruch class (or rather the equivariant Todd class) and the Poincar\'e series we will study elsewhere.

\subsection{Toric singularities}
If the group $G\subset GL_n(\C)$ is abelian, we may assume that the chosen representation is diagonal, so the action of $G$ commutes with the action of the torus $\TT=(\C^*)^n$. The quotient singularity is an affine toric variety. Its local equivariant Hirzebruch class  with respect to the action of the torus $\TT$ can be computed via a toric resolution. By Brion-Vergne \cite{BrVe} the local Todd class (i.e. for $y=0$) can be computed by counting lattice points in the dual cone. In fact the method of the proof in \cite{BrVe} is as in our case based on the localization to the fixed points, a version of Lefschetz-Rieman-Roch.
The generalization of Brion-Vergne result for the Hirzebruch class is given in \cite{We3}.

\begin{proposition} \label{toric formula}Let $X_\sigma$ be an affine toric variety given by the cone $\sigma$. Let $p$ be the fixed point of $X_\sigma$. Then
$$\frac{td_y^\TT(X_\sigma)_{|p}}{eu(M,p)}=\sum_{\tau\subset \sigma^\vee}(1+y)^{\dim(\tau)}\sum_{m
\in int(\tau)\cap \Lambda}e^{-m}\,.$$
Here we identify the lattice $\Lambda=Hom(\TT,\C^*)$ with $H^2_\TT(pt;\Z)$ and the summation is taken with respect to faces (of any dimension)  of the dual cone $\sigma^\vee$.\end{proposition}

\begin{remark}\rm When we restrict the action to one dimensional diagonal torus then the formula above  (at least when we set $y=0$)  reduces to computation of the classical Molien series: counting lattice points corresponds to counting the dimensions of $Sym^*(\C^n)^G$. This way we obtain another proof of Theorem \ref{th1} for diagonal representations.\end{remark}

Let us give an example of surface singularities with $G=\Z_n\subset SL_2(\C)$,
i.e.
$A_{n-1}$ with the action of $(\C^*)^2$. We set $T_k=e^{-t_k}$ for $k=1,\,2$. We have four ways of computing the Hirzebruch class and obtain four different expressions. We leave to the reader checking that these results are equal.

\smallskip

$\bullet$ The Hirzebruch class computed for $A_{n-1}$ as a hypersurface in $\C^3$:
$$(1 + y) (1 +
    T_1 T_2 y) \frac{1 - (T_1 T_2)^n}{(1 - T_1 T_2) (1 - T_1^n) (1 - T_2^n)} - y\,.$$

$\bullet$ The Hirzebruch class via resolution:
$$\sum_{i=0}^{n-1}\frac{1 +y\, T_1^{i + 1} T_2^{i + 1 - n}}{1 - T_1^{i + 1} T_2^{i + 1 - n}}\cdot\frac {1 +y\, T_1^{-i} T_2^{n - i}}{1 - T_1^{-i} T_2^{n - i}}+(n-1)y.$$

$\bullet$  The Hirzebruch class via counting lattice points:
$$1 + (y + 1)\left(\frac{T_1^n}{1 - T_1^n} +\frac
    {T_2^n}{1 - T_2^n}\right) + (y + 1)^2\frac1{1 - T_1^n}\frac1{1 - T_2^n} \sum_{k=1}^n(T_1 T_2)^k.$$

$\bullet$ The Hirzebruch class via Lefschetz-Riemann-Roch:
$$\frac 1n\sum_{k=0}^{n-1}\frac{1+y\,e^{\frac{2k\pi i}n} T_1}{1-\,e^{\frac{2k\pi i}n} T_1}\cdot\frac{1+y\,e^{-\frac{2k\pi i}n} T_2}{1-\,e^{-\frac{2k\pi i}n} T_2}.$$

\section{Hirzebruch class of a crepant resolution}\label{sec-crepant}
The elliptic genus was defined by many authors. We focus on the version of Borisov and Libgober. In  \cite{BoLi0} a historical account is given and different versions of the elliptic genus are discussed.
The elliptic genus generalizes the Hirzebruch class and behaves well with respect to crepant resolutions. First we review basic necessary constructions of Borisov and Libgober \cite{BoLi1, BoLi} and next we specialize the results of \cite{BoLi} to the Hirzebruch class of $\C^n/G$. Our goal is to give a formula for the Hirzebruch class of a crepant resolution in terms of Molien series.

\subsection{Elliptic genus}
Let us define the theta function\footnote{We give a definition according to \cite{Cha} but following  \cite{BoLi0,BoLi1,BoLi} we set $q=e^{2\pi i \tau}$
not $q=e^{\pi i \tau}$. Therefore we have to divide the exponents of $q$ by 2.
}
\begin{align*}\theta(\upsilon,\tau)&=\frac 1i\sum_{n=-\infty}^\infty(-1)^n q^{\frac12 (n+\frac12)^2}e^{(2n+1)\pi i\upsilon}=\\
&=2\sum_{n=0}^\infty(-1)^n q^{\frac12 (n+\frac12)^2}\sin((2n+1)\pi \upsilon)
\,,\end{align*}
where  $q=e^{2\pi i\tau}$,
see \cite[Ch.~V.1]{Cha}. The series is convergent for $Im(\tau)>0$ (i.e. $|q|<1$) and $\upsilon\in \C$, but we treat it only  as a power series in $\upsilon$ with a parameter $\tau$.
According to Jacobi product formula \cite[Ch~V.6]{Cha}
\begin{equation}\label{Jacobi-prod-for}
\theta(\upsilon,\tau)=q^{\frac18}\,2 \sin (\pi \upsilon)
\prod_{l=1}^{l=\infty}(1-q^l)
 \prod_{l=1}^{l=\infty}(1-q^l e^{2\pi i\upsilon} )(1-q^l e^{-2\pi i\upsilon})\,.
\end{equation}
For a smooth complex variety the {\it elliptic class} is defined in terms of the Chern roots $x_i$ of the tangent bundle as
\begin{equation}\label{elldef}
\mathcal{ELL}(Y;z,\tau)=
\prod_{k=1}^{\dim(Y)} x_k
\frac
{\theta(\frac{x_k}{2 \pi i}-z,\tau)}
{\theta(\frac{x_k}{2 \pi i }, \tau)}\in H^*(Y)\otimes\C[[z,\tau]]
\,.
\end{equation}
The {\it elliptic genus} is  the integral $$\int_Y\mathcal{ELL}(Y;z,\tau)\,.$$ Let us take the limit of the elliptic class
when $\tau \to i\infty$ (or when $q\to 0$). By the Jacobi product formula (\ref{Jacobi-prod-for}) we have
\begin{multline}\label{limit1}\lim_{\tau \to i\infty}
\frac{\theta(\frac x{2 \pi i}-z,\tau)}{\theta(\frac x{2 \pi i},\tau)}=\frac{\sin(\pi(\frac{x}{2 \pi i}-z))}{\sin(\pi\frac{x}{2 \pi i})}
=\\=\frac{e^{\pi i(\frac{x}{2 \pi i}-z)}-e^{-\pi i(\frac{x}{2 \pi i}-z)}}{e^{\pi i\frac{x}{2 \pi i}}-e^{-\pi i\frac{x}{2 \pi i}}}=
e^{-\pi iz}
\frac{(1-e^{2 \pi i z}e^{-x})}{(1-e^{-x})} \,.
\end{multline}
Therefore
$$\lim_{\tau \to i\infty}\mathcal{ELL}(Y;z,\tau)=
e^{-\dim(Y)\pi iz}td_{-e^{2 \pi i z}}(Y)\,,$$
which can be written as
$$(-y)^{-\frac{\dim(Y)}2}td_{y}(Y)\,,\quad\text{with } y=-e^{2 \pi i z}\,.$$
Note that in \cite{BoLi0,BoLi} $e^{2\pi i z}=y$, but we want to have a formula which agrees with our convention for $\chi_y$ genus, thus we introduce the minus sign.

Then a relative elliptic genus for Kawamata log-terminal
pairs $(Y,D)$ is introduced, for definition see \cite[Def. 3.7]{Bat} or \cite[\S2]{BoLi}. If $f:Y\to X$ is a resolution of a variety  with (at most) $\Q$-Gorenstein singularities and $D=K_Y-f^*K_X$, then the relative elliptic genus is independent of the resolution. % \cite[Prop. 3.6, 3.7]{BaLi1}.
This way one obtains an invariant of singular varieties, see \cite[Prop. 3.6, 3.7]{BoLi1}.
The construction is local and allows to define the characteristic class $\mathcal{ELL}(X;z,\tau)$ for varieties with $\Q$-Gorenstein singularities. If $f:Y\to X$ is a crepant resolution then
\begin{equation}\label{equivell}\mathcal{ELL}(X;z,\tau)=f_*(\mathcal{ELL}(Y;z,\tau))\in H_*(X)\otimes\C[[z,\tau]]\,.\end{equation}
In particular, $f_*(\mathcal{ELL}(Y;z,\tau))$ does not depend on the crepant resolution.

\begin{corollary}\label{niezaleznosc} Suppose $X$ has at most $\Q$-Gorenstein singularities. For a crepant resolution $f:Y\to X$ the push-forward $$f_*(td_y(Y))\in H_*(X)\otimes\C[y]$$
does not depend on $Y$.\end{corollary}

Note that for symplectic quotients of dimension 4 Corollary \ref{niezaleznosc} follows from  \cite[Thm 3.2]{AnWi} where it is shown that any two crepant resolutions differ by a sequence of flops.

\subsection{Orbifold elliptic genus}
Suppose a finite group $G$ acts on a complex manifold $Y$. For any two commuting elements $g,h\in G$ denote by $Y^{g,h}$ the fixed point set for both elements. For the sake of simplicity we assume that $Y^{g,h}$ is connected. For a pair of commuting elements $g,h$ we decompose the bundle $TY_{|Y^{g,h}}\simeq\bigoplus_\lambda V_\lambda$ into common eigen-subbundles. Let $x_\lambda$ be the first Chern class of $V_\lambda$. (We assume that
$\dim(V_\lambda)=1$, otherwise we use the splitting principle.)
Suppose that $g$ (resp. $h$) acts on $V_\lambda$ via multiplication by $e^{2\pi i \lambda(g)}$ with $\lambda(g)\in\Q\cap[0,1)$ (resp. by $e^{2\pi i \lambda(h)}$, $\lambda(h)\in\Q\cap[0,1)$).
The orbifold elliptic class is defined by the formula
\begin{multline}\label{ellorbdef}\mathcal{ELL}_{orb}(Y,G;z,\tau)=\\\frac1{|G|}
\sum_{gh=hg}(i_{Y^{g,h}})^*\left(
\prod_{\lambda(g)=\lambda(h)=0}x_{\lambda}
\prod_{\lambda} \frac{{\theta(\frac{x_\lambda} {2 \pi i}+\lambda(g)-\tau \lambda(h)-z,\tau)}}
{{\theta(\frac{x_\lambda} {2 \pi i }+\lambda(g)-\tau \lambda(h),\tau)}}e^{2\pi i \lambda(h)z}
\right),
\end{multline}
where $i_{Y^{g,h}}:Y^{g,h}\to Y$ is the inclusion.
Note\footnote{Instead of $\lambda(g)$ we should have written $\lambda^{g,h}_k(g)$ with $k=1,\dots,\dim(X^{g,h})$ and $\lambda^{g,h}_k(h)$ instead of $\lambda(h)$, but we do not want to make the formula complicated and we keep the notation of \cite{BoLi}.} that in the summation
 the numbers $\lambda(g)$ and $\lambda(h)$ in fact depend  on the pair $(g,h)$, because the decomposition of $\C^n$ into eigenspaces of $g$ has to be $h$-invariant.
The main result of \cite{BoLi} is the following
\begin{theorem}\label{lbthell}\cite[Th. 5.3]{BoLi}
 Let  $X=Y/G$ be a quotient variety with $Y$ smooth. Suppose $\pi^*(K_X)=K_Y$, where $\pi$ is the quotient map. Then
 $$\pi_*(\mathcal{ELL}_{orb}(Y,G;z,\tau))=\mathcal{ELL}(X;z,\tau)\in H_*(X)\otimes\C[[z,\tau]]\,.$$
Therefore if $f:\widetilde X\to X$ is a crepant resolution of the quotient variety, then
$$\pi_*(\mathcal{ELL}_{orb}(Y,G;z,\tau))=f_*(\mathcal{ELL}(\widetilde X;z,\tau))\,.$$
\end{theorem}
In fact Borisov and Libgober prove an equality for quotients of $G$-Kawamata log-terminal $G$-normal pairs (\cite[Def. 3.2]{BoLi}). The second equality given here follows from the birational invariance of the elliptic class (\cite[Th 3.7]{BoLi}).

\subsection{An equivariant version of the elliptic class}

If a torus $\T$ (or any other algebraic group) acts on a G-variety $Y$ and the actions commute, then one can define equivariant elliptic cohomology classes $\mathcal{ELL}^\T(Y,G;z,\tau)$ and $\mathcal{ELL}_{orb}^\T(\widetilde Y;z,\tau)$ in $\hat H^*_\T(Y)\otimes\C[[z,\tau]]$ by applying the formulas (\ref{elldef}) and (\ref{ellorbdef}) to equivariant tangent bundles. We approximate the Borel construction by $(\C^{m+1}\setminus \{0\})\times^\T Y$. The equivariant elliptic class is approximated by $$(p_m)^*\mathcal{ELL}(\P^{m};z,\tau)^{-1}\,\cap\,\mathcal{ELL}\big((\C^{m+1}\setminus \{0\})\times^\T Y\,,\,G;z,\tau\big)\,,
$$
where $p_m:(\C^{m+1}\setminus \{0\})\times^\T Y\to\P^{m}$ is the projection.
Theorem \ref{lbthell} is applied to the twisted product $\C^{m+1}\setminus \{0\})\times^\T Y$ and in the limit we obtain the equality  for the equivariant classes.

The straight-forward verification of the formula (\ref{equivell}) in the equivariant context was done in \cite{Wae1}. The equivariant counterpart of McKay correspondence for elliptic genus, i.e.~the equivariant version of  Theorem \ref{lbthell}, was proved in \cite{Wae2}. Recently the equivariant elliptic class in the context of equivariant localization is studied in \cite{Lib}.

\subsection{Comparison with Molien series}
Let $Y=\C^n$, $G\subset GL_n(\C)$.
 The group $\T$ is acting on $\C^n$ via scalar multiplication.  We study the equivariant version of the orbifold elliptic class and its limit when $\tau\to i\infty$.
We will show that the limit can be expressed by the extended Molien series of centralizers of elements of $G$. Let us introduce some notation.

For a group $H$ acting on a vector space $W$ let us denote by $Mol(H,W;v,T)$ the extended Molien series.

Recall that the age of an element
 $g\in  GL_n(\C)$ of finite order is defined as $\sum_{k=1}^n \lambda_k$, where $e^{2\pi i\lambda_k}$, $k=1,\dots,n$ are the eigenvalues of $g$ and $\lambda_k\in[0,1)\cup \Q$.

We will prove that
\begin{theorem}\label{bltheorem}Let $G\subset GL_n(\C)$. Then
$$\lim_{\tau \to i\infty}
\mathcal{ELL}^\T_{orb}(\C^n,G;z,\tau)=t^n
(-y)^{-\frac n2}\sum _{h\in Conj(G)}(-y)^{age(h)}Mol(C(h), (\C^n)^h;yT,T)\,,$$
where $T=e^{-t}$.
\end{theorem}
If $G\subset SL_n(\C)$, then $age(g)$ is an integer and
\begin{corollary} \label{blcorollary}Let $G\subset SL_n(\C)$ and let $f:\widetilde X\to X=\C^n/G\hookrightarrow M$ be a crepant resolution. Then
$$\frac{f_*td_y^\T(\widetilde X)}{eu(M,0)}=\sum_{h\in Conj(G)}(-y)^{age(h)}Mol(C(h),(\C^n)^h;yT,T)\,,$$
where $C(h)$ is the centralizer of $h$ in $G$ and $M$ is an ambient space containing $\C^n/G$.
\end{corollary}
\begin{proof} Having in mind that each $x_\lambda=t$ and that the action of $(i_{g,h})_*$ is the multiplication by $t^{\codim (Y^{g,h})}$ we rewrite the definition of the elliptic class
\begin{multline}\mathcal{ELL}^\T_{orb}(\C^n,G;z,\tau)=\\=\frac1{|G|}
\sum_{h\in G}\sum_{g\in C(g)}{t^n}
\prod_{\lambda} \frac{{\theta(\frac{t} {2 \pi i }+\lambda(g)-\tau \lambda(h)-z,\tau)}
}{
{\theta(\frac t {2 \pi i }+\lambda(g)-\tau \lambda(h),\tau)}}
e^{2\pi i  \lambda(h)z}
\,.
\end{multline}
Let us study the limit of the class $\mathcal{ELL}^\T_{orb}(\C^n,G;z,\tau)_{|0}$ when $\tau\to i\infty$.
First observe that
\begin{multline}
\lim_{\tau \to i\infty}
\frac
{\theta(a-\lambda\tau-z)}
{\theta(a-\lambda\tau)}
=
\lim_{\tau \to i\infty}  \frac{\sin(\pi(a-\lambda\tau-z))}{\sin(\pi(a-\lambda\tau))}
=\\=\lim_{s \to \infty}  \frac{e^{\pi i(a-\lambda i s-z)}-e^{-\pi i(a-\lambda i s -z)}}{e^{\pi i(a-\lambda i s)}-e^{-\pi i(a-\lambda i s)}}= \lim_{s \to \infty}
\frac{e^{\pi(\lambda s+ i(a-z))}-e^{-\pi(\lambda s+ i(a-z))}}{e^{\pi (\lambda s+ ia)}-e^{-\pi(\lambda s+ ia)}}=\\=e^{-\pi iz} \,.
\end{multline}
for $\lambda\in(0,1)$. Also, we will apply the equality (\ref{limit1}) with $\frac x{2\pi i}+\lambda(g)$ instead of $\frac x{2\pi i}$.
Therefore
\begin{multline}\lim_{\tau\to \infty}
\mathcal{ELL}^\T_{orb}(Y,G;z,\tau)=\\=\frac{t^n}{|G|}
\sum_{h\in G}\sum_{g\in C(g)}
\prod_{\lambda: \lambda(h)>0} e^{-\pi i z}
e^{2\pi i  \lambda(h)z}
\prod_{\lambda: \lambda(h)=0}
e^{-\pi iz}
\frac{1-e^{2 \pi i z}e^{-(t+ 2 \pi i \lambda(g))}}{1-e^{-(t+ {2 \pi i }\lambda(g))}}
e^{2\pi i  \lambda(h)z}
\,.
\end{multline}
Now setting $T=e^{-t}$, $e^{2\pi i z}=-y$, with the convention that $(-y)^{\frac12}=e^{\pi i z}$ we obtain
\begin{multline}\frac{t^n}{|G|}
\sum_{h\in G}\sum_{g\in C(g)} e^{-n\pi i z}
e^{2\pi i \sum_\lambda\lambda(h)z}
\prod_{\lambda: \lambda(h)=0}
\frac{1-e^{2 \pi i z}e^{-(t+ 2 \pi i \lambda(g))}}{1-e^{-(t+ {2 \pi i }\lambda(g))}}=\\=
\frac{t^n}{|G|}
\sum_{h\in G}\sum_{g\in C(g)}(-y)^{-\frac{n}2}
(-y)^{age(h)}
\prod_{\lambda: \lambda(h)=0}
\frac{1-y\, a_\lambda(g)^{-1} T}{1-a_\lambda(g)^{-1} T}\,.
\end{multline}
Here $a_\lambda(g)=e^{2\pi i \lambda(g)}$ is an eigenvalue of $g$. Finally we obtain
\begin{multline}\frac{t^n(-y)^{-\frac{n}2} }{|G|}
\sum_{h\in G}\sum_{g\in C(g)}
(-y)^{age(h)}  \frac{\det(Id+y T g_{|(\C^n)^h})}{\det(Id-T g_{|(\C^n)^h})}=\\={t^n(-y)^{-\frac{n}2}}\sum_{[h]\in Conj(G)}\frac1{|C(g)|}\sum_{g\in C(g)}
(-y)^{age(h)}  \frac{\det(Id+y T g_{|(\C^n)^h})}{\det(Id-T g_{|(\C^n)^h})}=\\=
{t^n}(-y)^{-\frac{n}2} \sum_{[h]\in Conj(G)}
(-y)^{age(h)}
Mol(C(h),(\C^n)^h;yT,T)\,.
\end{multline}
\end{proof}

Theorem \ref{bltheorem} can be interpreted as
$$\frac1{eu(M,0)}{f_*td_y^\T(\widetilde X)}
=\frac1{eu(M,0)}{f_*td_y^\T(\widehat{\C^n/G})}\,,$$
where
$$\widehat{\C^n/G}=\bigsqcup_{h\in Conj(G)} (\C^n)^h/C(h)\times \C^{age(h)}\,.$$
This space only slightly differs from the so-called inertia stack
$$\bigsqcup_{h\in Conj(G)} (\C^n)^h/C(h)=\{(x,g)\in \C^n\times G\;|\;gx=x\}/G\,,$$
defined already in \cite{BaumConnes}.

\subsection{Divisibility}
Denote by $\widetilde H(y,T)$ the sum
\begin{equation}\label{hzfalka}\sum_{[h]\in Conj(G)}
(-y)^{age(h)}
Mol(C(h),(\C^n)^h;yT,T)\,.\end{equation} By Theorem \ref{bltheorem}
$$\frac{f_*td_y^\T(\widetilde X)}{eu(M,0)}=\widetilde H(y,T)$$
for a crepant resolution of a quotient variety $\C^n/G$. By \cite{Wbb}
$$\lim_{T\to0}\frac{f_*td_y^\T(\widetilde X)}{eu(M,0)}=\chi_y(f^{-1}(0))\,.$$
If $f$ is crepant, then this limit is equal to $\widetilde H(y,0)=\sum_{g\in Conj(G)} (-y)^{age(g)}$. By \cite[Thm 8.4]{Bat} the cohomology of $f^{-1}(0)$ is pure of Hodge type $(k,k)$. It follows that $$\widetilde H(-x,0)=\sum b_{2k}(f^{-1}(0))x^k\quad\text{and}\quad \widetilde H(-1,0)=\chi_{top}(f^{-1}(0))=|Conj(G)|\,.$$

\begin{proposition} If $G\subset Sp_n(\C)$, then
$$\widetilde H(y,T)-(-y)^n \widetilde H(-1,0)$$
is divisible by $(y+1)(1+T^2y)$.
\end{proposition}

\begin{proof}  By Proposition \ref{symp-postac} $$(-y)^{age(h)}Mol(C(h),(\C^{2n})^h,vT,T)-(-y)^{age(h)+\frac12 {\dim((\C^{2n})^h}}$$ is divisible by $(1+y)(1+T^2y)$. For symplectic actions $age(h)=\frac 12\codim((\C^{2n})^h)$, see \cite{KaledinMcKay}. The exponent of $(-y)$ is equal to
$${\frac{\codim((\C^{2n})^h}2+\frac{\dim((\C^{2n})^h}2}=n\,.$$
To have divisibility of $\widetilde H(y,T)$ we have to subtract $(-y)^n$ for each  summand of (\ref{hzfalka}), i.e. for each conjugacy class $[h]\in Conj(G)$.
\end{proof}

\section{Final remarks, positivity}\label{final}
 The initial work of the second author \cite{We3} was directed towards the search of positivity results.
We can check  that in our examples (and in many others) after the substitution $T:=1+S$ and $y:=-1-\delta$  the numerator of $H(y,T)$ is a polynomial with nonnegative coefficients. That is always the case for simplicial toric varieties by \cite[Theorem 13.1]{We3}. The proof is based on the formula   for the Hirzebruch class of a toric variety,  Proposition \ref{toric formula}. Any representation of an abelian finite group can be diagonalized, therefore the quotient variety admits an action of the full torus. Hence such quotient is a simplicial toric variety and the positivity holds. For general quotient varieties we have no proof, except from a partial result, Proposition \ref{dodatniosc} and Corollary \ref{cordod}.

\begin{proposition}\label{dodatniosc}
If $G \subset Sp_{2n}(\C)$ then after the substitution $$T:=1+S\quad\text{and}\quad y:=-1-\delta$$ the Hirzebruch class $H(y,T)$ for $\C^n/G$ can be written as a quotient of polynomials with nonnegative coefficients. The polynomial in the denominator has roots in the unit circle.
\end{proposition}

\begin{proof}
We look at component of $H(y,T)$ corresponding to $g \in G$. If $\eps$ is an eigenvalue of a matrix $g \in G$ then, since $g$ is symplectic, $\overline{\eps} = \eps^{-1}$ is also an eigenvalue of $g$. Moreover, eigenvalues 1 and $-1$ appear with even multiplicities. Thus we may write
\begin{multline*}
\frac{P_g(y,T)}{Q_g(T)} = \prod_{k=1}^{2n}\frac{1+a_k(g)yT}{1-a_k(g) T} \\= \frac{\left(\prod_{k=1}^{n_1}(1+a_k(g)yT)(1+\overline{a_k(g)}yT)\right)(1+yT)^{2n_2}(1-yT)^{2n_3}}{\left(\prod_{k=1}^{n_1}(1-a_k(g)T)(1-\overline{a_k(g)}T)\right)(1+T)^{2n_2}(1-T)^{2n_3}}.
\end{multline*}
After given substitutions $(1-yT)^2$, $(1+yT)^2$, $(1+T)^2$, $(1-T)^2$ have nonnegative coefficients as polynomials in $\delta, S$. Also, for $\eps$ from the unit circle and a polynomial $P$ we have
\begin{multline*}
(1-\eps (1+P))(1-\overline{\eps}(1+P)) = 1 - (\eps + \overline{\eps})(1+P) + (1+P)^2 =\\= 2+2P - (\eps + \overline{\eps})(1+P) + P^2 = (2-\eps - \overline{\eps})(1+P) + P^2,
\end{multline*}
which has nonnegative coefficients if $P$ has. This shows that for any eigenvalue $a_k(g)$ both $(1+a_k(g)yT)(1+\overline{a_k(g)}yT)$ and $(1-a_k(g)T)(1-\overline{a_k(g)}T)$ have nonnegative coefficients in $\delta, S$.

Thus by formula (\ref{molien}) and Theorem \ref{th1} we see that
\begin{equation*}
H(y,T) = \sum_{g\in G} \frac{P_g(y,T)}{Q_g(T)}
\end{equation*}
is a sum of fractions where $P_g$ and $Q_g$ are products of indecomposable (over $\R$) factors, each factor has nonnegative coefficients after the considered substitutions. When we reduce the sum of fractions to the common denominator, then both numerator and denominator will have nonnegative coefficients after the substitutions.
\end{proof}

Multiplying the numerator and the denominator by $(1-\eps T)(1-\overline\eps T)$, where $\eps$ is a root of unity, we can achieve the product of the factors $1-T^{k_i}$ in the denominator, not loosing the positivity of the numerator.

\begin{corollary}\label{cordod}If $G \subset Sp_{2n}(\C)$ then the Hirzebruch class  for $\C^n/G$ can be written as
$$H(y,T)=\frac{P(y,T)}{\prod_{i=1}^r (1-T^{k_i})}\,,$$
where $P(-1-\delta,1+S)$ is  a  polynomial in $S$ and $\delta$ with nonnegative coefficients.
\end{corollary}

Equivariant Hirzebruch specializes to  Chern-Schwartz-MacPherson class (see \cite[\S1]{BSY}). For an equivariant embedding $i:X\hookrightarrow M$ the the image $i_*(c^\T_{SM})(X)$ is a polynomial in $t\in H^2\T(pt)$ of degree $\dim(M)$ and divided by $eu(M,0)$ does not depend on $M$. We have
\begin{equation}\label{granicach}\lim_{\delta\to 0}H(-1-\delta,e^{\delta t})=\frac{i_*(c^\T_{SM})(X)}{eu(M,0)}\,.\end{equation}

\begin{proposition}If $G \subset GL_{n}(\C)$ then the equivariant Chern-Schwartz-MacPherson class is equal to
$$i_*(c^\T_{SM})(X)=
\frac{eu(M,0)}{|G|}\sum_{g\in G}\left(\frac{1+t}t\right)^{\dim((\C^n)^g)}\in H^*_\T(M)\simeq \Q[t]\,.$$
\end{proposition}

\begin{proof} First note that
\begin{equation}\label{granica}\lim_{\delta \to 0}\frac{1-a(1+\delta)e^{\delta t}}{1-a\, e^{\delta t}}
=\left\{\begin{matrix}\frac{1+t}t&\text{for }a=1\\ \\1 &\text{for }a\not=1&\!\!\!.\end{matrix}\right. \end{equation}

To compute the limit (\ref{granicach}) we apply  formula (\ref{molien}) and Theorem \ref{th1}
\begin{multline}\frac{i_*(c^\T_{SM})(X)}{eu(M,0)}=\lim_{\delta\to 0}H(-1-\delta,e^{\delta t})=
\frac1{|G|}\sum_{g\in G}\lim_{\delta\to 0}\frac{\det(I-(1+\delta)e^{\delta t}g)}{\det(I-e^{\delta t}g)}=\\=\frac1{|G|}\sum_{g\in G}\prod_{k=1}^n\lim_{\delta\to 0}\frac{1-a_k(g)(1+\delta)e^{\delta t}}{1-a_k(g)e^{\delta t}}=
\frac1{|G|}\sum_{g\in G}\left(\frac{1+t}t\right)^{\dim((\C^n)^g)}\,.\end{multline}
\end{proof}

\begin{corollary}\label{cordodch}If $G \subset GL_{n}(\C)$ then the equivariant Chern-Schwartz-MacPherson class
$$i_*(c^\T_{SM})(X)\in H^*_\T(M)\simeq   \Q[t]$$
for $X=\C^n/G$  is  a  polynomial in $t$  with nonnegative coefficients.
\end{corollary}

This is a local version of a positivity property, which was studied for Schubert varieties in \cite{AlMi} and for hyperplane arrangements \cite[\S6]{Al-pos}.

\begin{remark}\rm The formulas  (\ref{hirquo}) and  (\ref{granica}) imply the analogous statement for  Chern-Schwartz-MacPherson classes of the global quotient $X=Y/G$, where $Y$ is a smooth variety, and $G$ is a finite group of automorphisms, see \cite[formula (13)]{CMSS}:
$$c_{SM}(X)=\frac1{|G|}\sum_{g\in G}(\pi_g)_*(c_{SM}(Y^g))\,.$$
Similar formulas hold for the equivariant global case. By Theorem \ref{bltheorem} for  crepant resolutions $f:\widetilde X\to X$
\begin{multline*}f_*c_{SM}(\widetilde X)=\sum_{[h]\in Conj(G)}c_{SM}(Y^h/C(h))=\\
=\sum_{[h]\in  Conj(G)}\frac1{|C[h]|}\sum_{g\in C(h)}(\pi_{g,h})_*(c_{SM}(Y^{g,h}))=\\
=\frac1{|G|}\sum_{gh=hg}(\pi_{g,h})_*(c_{SM}(Y^{g,h}))
\,. \end{multline*}
Here  $\pi_{g,h}:Y^{g,h}\to Y/G$ is the projection restricted to the fixed point set $Y^{g,h}$.
\end{remark}

The starting point of our common research was to study the Hirzebruch class from the point of view of existence of symplectic resolution. We have observed certain regularities, especially for
$\widetilde H(y,T)$ of quotients having a crepant resolution. However it is hard to grasp a general pattern. We hope it is possible to find a necessary criterion for existence of a crepant resolution.
This problem might be an interesting subject of further research.

\section{Appendix: more examples}
\label{sectionexamples}
If $G\subset Sp_n(\C)$ then  the crepant resolution of $\C^n/G$ is the same as symplectic resolution, see~\cite{Verbitsky}. All the examples given below are quotients by symplectic groups, and all these groups except example~(9) are generated by symplectic reflections (i.e. matrices such that the eigenspace for the eigenvalue 1 is of dimension $n-2$). Again by~\cite{Verbitsky}, this last condition is a necessary one for the existence of a symplectic resolution. Note that all irreducible matrix groups generated by symplectic reflections can be found in Cohen's classification~\cite{CohenReflections}.

For the sake of simplicity most of our examples are 4-dimensional. We give $H(y,T)=Mol(yT,T)$ and $\widetilde H(y,T)$ in the simplified fraction form. The computations presented here were performed with hope to discover whether the form of the equivariant Hirzebruch class is related to certain properties of the given quotient or its resolutions.

\subsection{List of tested groups}
\begin{enumerate}[leftmargin=*]
\item \textbf{Du Val singularities}. 
\item \textbf{A 4-dimensional symplectic group with 32 elements}, which is isomorphic to $Q_8 \times_{\mathbb{Z}_2} D_8$, where $Q_8$ is the quaternion group and $D_8$ is the dihedral group of 8 elements. The group is the first element of the second infinite series in the first part of Cohen's classification~\cite[Table~I]{CohenReflections}. By~\cite{BellamySchedler} symplectic resolutions of the corresponding quotient singularity exist, they were constructed in~\cite{DBW_81res}.

\item \textbf{A reducible 4-dimensional symplectic representation of the binary tetrahedral group.} It is obtained from a 2-dimensional representation by taking direct sum with its contragradient representation. There are 2 symplectic resolutions for the corresponding quotient singularity, they were constructed in~\cite{LehnSorger}.

\item \textbf{A 4-dimensional symplectic representation of the dihedral group $D8$}, constructed as a wreath product $\Z_2\wr S_2$.
This is an element of an infinite series of (reducible) symplectic representations, for which a symplectic resolution can be constructed using a suitable Hilbert scheme.

\item \textbf{A 4-dimensional symplectic representation of the symmetric group $S_3$.} This is the only 4-dimensional element of another infinite series of symplectic representations (of symmetric groups~$S_n$) for which symplectic resolutions come from a Hilbert scheme construction.

\item \textbf{A 4-dimensional symplectic group of order 16}, a semidirect product $(\mathbb{Z}_4 \times \mathbb{Z}_2) \rtimes \mathbb{Z}_2$. It appears for $m=2$ in the 7th infinite series in~\cite[Table~I]{CohenReflections}. It is a subgroup of the 32-element group~(2). It is not known whether a symplectic resolution exists.

\item \textbf{A 4-dimensional symplectic group of order 24}, a semidirect product $(\mathbb{Z}_6 \times \mathbb{Z}_2) \rtimes \mathbb{Z}_2$. It appears for $m=3$ in the 7th series in~\cite[Table~I]{CohenReflections}. It is not known whether a symplectic resolution exists.

\item \textbf{A 4-dimensional group of order 64}, isomorphic to $(\mathbb{Z}_4 \rtimes Q_8) \rtimes \mathbb{Z}_2$. It appears in the 1st series (for $m=2$) in~\cite[Table~I]{CohenReflections}, in particular it is generated by symplectic reflections. However, by~\cite{BellamySchedler} symplectic resolutions do not exist.

\item \textbf{The smallest (imprimitive) group generated by symplectic reflection in dimension 6.} It is a representation of the symmetric group $S_4$, constructed from $\Z_2$ and the trivial group as described in~\cite[Not.~2.8]{CohenReflections}. The symplectic resolution does not exist by~\cite[Thm~7.2]{BellamySchedler}.

\item \textbf{A representation of $\mathbb{Z}_5$ without symplectic reflections.} The considered subgroup $\Z_5\subset SL_4(\C)$ is generated by $diag (\xi^1,\xi^2,\xi^3,\xi^4)$. Note that this is not an action by symplectic reflections and therefore there is no symplectic, i.e. crepant, resolution.

\end{enumerate}

\subsection{Results for $H(y,T)$ and $\widetilde{H}(y,T)$}
\def\dolewej{\hspace*{-1cm}}

\begin{enumerate}[leftmargin=*]
\item The Molien series of Du Val singularities are given in \S\ref{duvallist}.
The Molien series of the crepant resolutions are obtained by subtracting $ky$, where $k$ is the number of the components of the exceptional divisor. 
\medskip

\item $Q_8 \times_{\mathbb{Z}_2} D_8$, order~32, dimension~4.
{\small
\vspace{-0.25cm}
\begin{multline*}\dolewej
H(y,T)=y^2+\frac{(y+1) (T^2 y+1)}{\left(1-T^2
\right)^2 \left(1-T^4\right)^2}\cdot\Big(
\left(T^8-2 T^6+4 T^4-2 T^2+1\right) \left(T^2 y^2+1\right)+\\+\left(-T^{10}+2 T^8+T^6+T^4+2 T^2-1\right) y\Big)
\end{multline*}
\vspace{-0.5cm}
\begin{multline*}\dolewej \widetilde H(y,T)=17y^2+\frac{(1 + y) (1 + T^2 y)}{(1 - T^2)^2 (1 - T^4)^2}\cdot \Big(
(1 - 2 T^2 + 4 T^4 - 2 T^6 + T^8)\left(T^2 y^2+1\right) -\\
 -2 \left(T^2+1\right) \left(3 T^8-9 T^6+11 T^4-9 T^2+3\right) y\Big)
\end{multline*}
\vspace{-0.25cm}
}

\item Binary tetrahedral group, dimension~4.
{\small
\vspace{-0.25cm}
\begin{multline*}\dolewej H(y,T)= y^2+\frac{(y+1) \left(T^2 y+1\right)}{(1-T^2)(1 - T^4)^2   (1 - T^6)} \cdot\Big(
\left(T^{12}+2 T^8+2 T^6+2 T^4+1\right) \left(T^2 y^2+1\right)+\\+\left(-T^{14}+T^{12}+4 T^{10}+4 T^8+4 T^6+4 T^4+T^2-1\right) y\Big)\end{multline*}
\vspace{-0.5cm}
\begin{multline*}\dolewej
\widetilde{H}(y,T) = 7y^2 + \frac{(y+1) \left(T^2 y+1\right)}{(1-T^2)(1 - T^4)^2   (1 - T^6)}
\cdot\Big(
\left(T^{12}+2T^8+2T^6+2T^4+1\right)(T^2y^2+1)-\\-
\left(T^2+1\right)\left(3T^{12}-4T^8-6T^6-4T^4+3\right)y
\Big)
\end{multline*}
\vspace{-0.25cm}
}

\item Dihedral group $D_8$, dimension~4.
{\small
\vspace{-0.25cm}
\begin{multline*}\dolewej
H(y,T)= y^2+
\frac{(y+1) \left(T^2 y+1\right)}{\left(1-T^2\right)^2 \left(1-T^4\right)^2}\cdot\Big(
\left(T^8+T^6+4 T^4+T^2+1\right) \left(T^2 y^2+1\right)+\\+\left(-T^{10}+2 T^8+7 T^6+7 T^4+2 T^2-1\right) y
\Big)
\end{multline*}
\vspace{-0.5cm}
\begin{multline*}\dolewej
\widetilde{H}(y,T) = 5y^2 + \frac{(y+1)\left(T^2y+1\right)}{\left(1-T^2\right)^2 \left(1-T^4\right)^2}\cdot\Big(
(T^8+T^6+4T^4+T^2+1)(T^2y^2+1)-\\-\left(T^2+1\right) \left(3 T^8-3 T^6-8 T^4-3 T^2+3\right)
y
\Big)
\end{multline*}
\vspace{-0.25cm}
}

\item Symmetric group $S3$, dimension~4.
{\small
\vspace{-0.25cm}
\begin{multline*}\dolewej H(y,T)= y^2+
\frac{(y+1) \left(T^2 y+1\right)}{(1-T^2)^2  (1-T^3)^2} \cdot\Big(
\left(T^6+T^4+2 T^3+T^2+1\right) \left(T^2 y^2+1\right)+\\+\left(-T^8+2 T^6+4 T^5+2 T^4+4 T^3+2 T^2-1\right) y
\Big)
\end{multline*}
\vspace{-0.5cm}
\begin{multline*}\dolewej
\widetilde{H}(y,T) = 3y^2+
\frac{(y+1)(T^2y+1)}{(1-T^2)^2(1-T^3)^2}
\cdot\Big(
(T^6+T^4+2T^3+T^2+1)(T^2y^2+1)-\\
 -\left(T^2+1\right) \left(2 T^6+2 T^5-3 T^4-8 T^3-3 T^2+2
T+2\right)
y
\Big)
\end{multline*}
\vspace{-0.25cm}
}

\item $(\mathbb{Z}_4 \times \mathbb{Z}_2) \rtimes \mathbb{Z}_2$, order~16, dimension~4.
{\small
\vspace{-0.25cm}
\begin{multline*}\dolewej
H(y,T)= y^2 +
\frac{(y+1) \left(T^2 y+1\right)}{\left(1-T^2\right)^2 \left(1-T^4\right)^2} \cdot\Big(
\left(T^8-T^6+4 T^4-T^2+1\right) \left(T^2 y^2+1\right)
+\\+
\left(-T^{10}+2 T^8+3 T^6+3 T^4+2 T^2-1\right) y
\Big)
\end{multline*}
\vspace{-0.5cm}
\begin{multline*}\dolewej
\widetilde{H}(y,T) = 10y^2 +
\frac{(y+1) \left(T^2 y+1\right)}{\left(1-T^2\right)^2 \left(1-T^4\right)^2}\cdot\Big(
(T^8-T^6+4T^4-T^2+1)(T^2y^2+1)-\\-\left(T^2+1\right) \left(4 T^8-9 T^6+6 T^4-9 T^2+4\right)y
\Big)
\end{multline*}
\vspace{-0.25cm}
}

\item $(\mathbb{Z}_6 \times \mathbb{Z}_2) \rtimes \mathbb{Z}_2$, order~24, dimension~4.
{\small
\vspace{-0.25cm}
\begin{multline*}\dolewej
H(y,T)= y^2+
\frac{(y+1) \left(T^2 y+1\right)}{(1-T^4)^2  (1-T^6)^2}
\cdot\\ \cdot\Big(\left(T^{16}+T^{14}+2 T^{12}+4 T^{10}+8 T^8+4 T^6+2 T^4+T^2+1\right) \left(T^2 y^2+1\right)
+\\+
\left(-T^{18}+3 T^{14}+8 T^{12}+14 T^{10}+14 T^8+8 T^6+3 T^4-1\right) y
\Big)
\end{multline*}
\vspace{-0.5cm}
\begin{multline*}\dolewej
\widetilde{H}(y,T) = 9y^2 +
\frac{(y+1) \left(T^2 y+1\right)}{(1-T^4)^2  (1-T^6)^2}
\cdot\\\Big(
(T^{16}+T^{14}+2T^{12}+4T^{10}+8T^8+4T^6+2T^4+T^2+1)(T^2y^2+1)-\\
-\left(T^2+1\right) \left(3 T^{16}+T^{14}-2 T^{12}-8 T^{10}-12
T^8-8 T^6-2 T^4+T^2+3\right)
y
\Big)
\end{multline*}
\vspace{-0.25cm}
}

\item $(\mathbb{Z}_4 \rtimes Q_8) \rtimes \mathbb{Z}_2$, order~64, dimension~4.
{\small
\vspace{-0.25cm}
\begin{multline*}\dolewej
H(y,T)= y^2+
\frac{(y+1) \left(T^2 y+1\right)}{\left(1-T^2\right)^2 \left(1-T^4\right) \left(1-T^8\right)}
\cdot\\ \cdot\Big(\left(T^{12}-2 T^{10}+3 T^8-2 T^6+3 T^4-2 T^2+1\right) \left(T^2 y^2+1\right)
+\\+
\left(-T^{14}+2 T^{12}+T^8+T^6+2 T^2-1\right) y
\Big)
\end{multline*}
\vspace{-0.5cm}
\begin{multline*}\dolewej
\widetilde{H}(y,T) = 16y^2 +
\frac{(y+1) \left(T^2 y+1\right)}{\left(1-T^2\right)^2 \left(1-T^4\right) \left(1-T^8\right)}
\cdot\\ \cdot\Big(\left(T^{12}-2T^{10}+3T^8-2T^6+3T^4-2T^2+1\right)\left(T^2y^2+1\right)-\\
-\left(T^2+1\right) \left(5 T^{12}-14 T^{10}+21 T^8-26 T^6+21
T^4-14 T^2+5\right)y
\Big)
\end{multline*}
\vspace{-0.25cm}
}

\item Symmetric group $S_4$, dimension~6.
{\small
\vspace{-0.25cm}
\begin{multline*}\dolewej
H(y,T)= -y^3+
\frac{(y+1)(T^2y+1)}
{(T^4-1)^2(T^3-1)^2(T^2-1)^2}\cdot\\
\cdot
\Big((T^{12}+T^{10}+2T^9+4T^8+2T^7+4T^6+2T^5+4T^4+2T^3+T^2+1)(y^4T^4+1)+\\
+(-T^{14}+2T^{12}+4T^{11}+7T^{10}+10T^9+16T^8+20T^7+16T^6+\\\hfill+10T^5+7T^4+4T^3+2T^2-1)(y^2T^2+1)y+\\
+(T^{16}-T^{14}-2T^{13}+12T^{11}+19T^{10}+30T^9+26T^8+30T^7+19T^6+\\+12T^5-2T^3-T^2+1)y^2\Big)
\end{multline*}
\vspace{-0.5cm}
\begin{multline*}\dolewej
\widetilde{H}(y,T) =
 -5y^3+
\frac{(y+1)(T^2y+1)}
{(T^4-1)^2(T^3-1)^2(T^2-1)^2}\cdot\\
\cdot\Big(
(T^{12}+T^{10}+2 T^9+4 T^8+2 T^7+4 T^6+2 T^5+4 T^4+2 T^3+T^2+1)(T^4 y^4+1)-\\
-(T+1)^2 (2 T^{12}-2 T^{11}+4
   T^{10}-6 T^9+3 T^8-12 T^7-2 T^6-12 T^5+3 T^4-6 T^3+\\\hfill+4 T^2-2 T+2)
(T^2y^2+1)y+\\
+2  (T^2+T+1 )
   (2 T^{14}-T^{13}-2 T^{12}-5 T^{11}-4 T^{10}+6 T^9+8 T^8+16 T^7+8
T^6+\\\hfill+6 T^5-4 T^4-5 T^3-2
   T^2-T+2) y^2
\Big)
\end{multline*}
\vspace{-0.25cm}
}

\item $\mathbb{Z}_5$, no symplectic reflections, dimension~4.
{\small
\vspace{-0.25cm}
\begin{multline*}\dolewej
H(y,T)= y^2+\frac{(y+1) \left(yT^2 +1\right)}{(1-T)^3
   \left(1-T^5\right)}\cdot\Big(
 (y^2T^2+1) (1 - 3 T + 5 T^2 - 3 T^3 + T^4)-\\
 -(1 - 3 T + 2 T^2 - 2 T^3 + 2 T^4 - 3 T^5 + T^6) y\Big)
\end{multline*}
\vspace{-0.5cm}
\begin{multline*}\dolewej 
\widetilde{H}(y,T) =
5y^2+\frac{(y+1) \left(yT^2 +1\right)}{(1-T)^3
   \left(1-T^5\right)}\cdot\Big(
\left(T^4-3T^3+5T^2-3T+1\right)\left(T^2y^2+1\right)-\\
-\left(T^6-3T^5+2T^4-2T^3+2T^2-3T+1\right)y
\Big)
\end{multline*}
\vspace{-0.25cm}
}

\end{enumerate}

\section{Appendix: Extended Molien series}\label{extmolien}

We prove formula (\ref{molien}).

For a vector space $V$ denote by $A(V)$ the bigraded vector space $$\sum_{k=0}^\infty\sum_{\ell=0}^{\dim (V)}Sym^k(V)\otimes \Lambda^\ell V\,.$$ We have \begin{equation}\label{mult}A(V\oplus W)=A(V)\otimes A(W)\,.\end{equation}
Let $g:V\to V$ a  linear map.
Let $ \tilde g:A(V)\to A(V)$ be the induced map.
Denote by $tr_*(g)$ the generating function of traces \begin{equation}\sum_{k,\ell}tr(\tilde g_{|Sym^k(V)\otimes \Lambda^\ell V})T^kv^\ell\end{equation}
If $\dim( V)=1$ and
 $g$ is the multiplication by $a$, then
\begin{equation}\label{jedim}tr_*( g)=\sum_{k=0}^\infty\sum_{\ell=0}^1 a^{k+\ell}T^kv^\ell=\frac{1+av}{1-aT}\,.\end{equation}
For an automorphism $g:V\to V$ which is  semisimple\footnote{Semisimplicity can be dropped.} on $V$ by (\ref{mult}) and (\ref{jedim})
we have
$$tr_*(g)=\prod_{i=1}^{\dim( V)}\frac{1+a_iv}{1-a_iT}\,, $$
where $a_i$ for $i=1,\dots ,\dim(V)$ are the eigenvalues of $g$. Hence
$$tr_*(g)=\frac{\det(1+vg)}{\det(1-Tg)}\,.$$
To compute the extended Molien series we use the well known formula
$$\dim(V^G)=\frac1{|G|}\sum_{g\in G}tr( g)\,,$$
which applied to every summand $Sym^k(V)\otimes \Lambda^\ell V$ instead of $V$ gives us
$$Mol(v,T)=\frac1{|G|}\sum_{g\in G}tr_*( g)\,.$$
We obtain
$$Mol(v,T)=\frac1{|G|}\sum_{g\in G}\frac{\det(1+vg)}{\det(1-Tg)}
\,,$$
where $a_i(g)$ for $i=1,\dots ,\dim(V)$ are the eigenvalues of $g\in G$ acting on~$V$. If the action of $G$ is replaced by the dual action on $V^*$ then the Molien series remains unchanged, since the eigenvalues of $g$ are the same as the eigenvalues of $(g^*)^{-1}$.

%Zeby zrobic automatycznie bibliografie odprocentuj nastepna linie
%\bibliographystyle{alpha}\bibliography{hirmolien}\end{document}

\begin{thebibliography}{CDGZ04}

\bibitem[AB84]{AtBo2}
Michael~F. {Atiyah} and Raoul {Bott}.
\newblock {The moment map and equivariant cohomology.}
\newblock {\em {Topology}}, 23:1--28, 1984.

\bibitem[Alu13]{Al-pos}
Paolo Aluffi.
\newblock Grothendieck classes and {C}hern classes of hyperplane arrangements.
\newblock {\em Int. Math. Res. Not. IMRN}, (8):1873--1900, 2013.

\bibitem[AM09]{AlMi}
Paolo Aluffi and Leonardo~Constantin Mihalcea.
\newblock Chern classes of {S}chubert cells and varieties.
\newblock {\em J. Algebraic Geom.}, 18(1):63--100, 2009.

\bibitem[AW14]{AnWi}
Marco Andreatta and Jaros{\l}aw~A. Wi{\'s}niewski.
\newblock 4-dimensional symplectic contractions.
\newblock {\em Geom. Dedicata}, 168:311--337, 2014.

\bibitem[Bat99]{Bat}
Victor~V. Batyrev.
\newblock Non-{A}rchimedean integrals and stringy {E}uler numbers of
  log-terminal pairs.
\newblock {\em J. Eur. Math. Soc. (JEMS)}, 1(1):5--33, 1999.

\bibitem[Bau82]{Ba}
Paul Baum.
\newblock Fixed point formula for singular varieties.
\newblock In {\em Current trends in algebraic topology, {P}art 2 ({L}ondon,
  {O}nt., 1981)}, volume~2 of {\em CMS Conf. Proc.}, pages 3--22. Amer. Math.
  Soc., Providence, R.I., 1982.

\bibitem[BC88]{BaumConnes}
Paul Baum and Alain Connes.
\newblock Chern character for discrete groups.
\newblock In {\em A f\^ete of topology}, pages 163--232. Academic Press,
  Boston, MA, 1988.

\bibitem[BFQ79]{BFQ}
Paul Baum, William Fulton, and George Quart.
\newblock Lefschetz-{R}iemann-{R}och for singular varieties.
\newblock {\em Acta Math.}, 143(3-4):193--211, 1979.

\bibitem[BKR01]{BridgelandKingReid}
Tom Bridgeland, Alastair King, and Miles Reid.
\newblock The {M}c{K}ay correspondence as an equivalence of derived categories.
\newblock {\em J. Amer. Math. Soc.}, 14(3):535--554 (electronic), 2001.

\bibitem[BL00]{BoLi0}
Lev~A. Borisov and Anatoly Libgober.
\newblock Elliptic genera of toric varieties and applications to mirror
  symmetry.
\newblock {\em Invent. Math.}, 140(2):453--485, 2000.

\bibitem[BL03]{BoLi1}
Lev Borisov and Anatoly Libgober.
\newblock Elliptic genera of singular varieties.
\newblock {\em Duke Math. J.}, 116(2):319--351, 2003.

\bibitem[BL05]{BoLi}
Lev Borisov and Anatoly Libgober.
\newblock Mc{K}ay correspondence for elliptic genera.
\newblock {\em Ann. of Math. (2)}, 161(3):1521--1569, 2005.

\bibitem[BS13]{BellamySchedler}
G.~{Bellamy} and T.~{Schedler}.
\newblock {On the (non)existence of symplectic resolutions for imprimitive
  symplectic reflection groups}.
\newblock {\em ArXiv e-prints}, September 2013.

\bibitem[BSY10]{BSY}
Jean-Paul Brasselet, J{\"o}rg Sch{\"u}rmann, and Shoji Yokura.
\newblock Hirzebruch classes and motivic {C}hern classes for singular spaces.
\newblock {\em J. Topol. Anal.}, 2(1):1--55, 2010.

\bibitem[BV82]{BeVe}
Nicole Berline and Mich{\`e}le Vergne.
\newblock Classes caract\'eristiques \'equivariantes. {F}ormule de localisation
  en cohomologie \'equivariante.
\newblock {\em C. R. Acad. Sci. Paris S\'er. I Math.}, 295(9):539--541, 1982.

\bibitem[BV97]{BrVe}
Michel Brion and Mich{\`e}le Vergne.
\newblock An equivariant {R}iemann-{R}och theorem for complete, simplicial
  toric varieties.
\newblock {\em J. Reine Angew. Math.}, 482:67--92, 1997.

\bibitem[CDGZ04]{CDG}
A.~Campillo, F.~Delgado, and S.~M. Gusein-Zade.
\newblock Poincar\'e series of a rational surface singularity.
\newblock {\em Invent. Math.}, 155(1):41--53, 2004.

\bibitem[Cha85]{Cha}
K.~Chandrasekharan.
\newblock {\em Elliptic functions}, volume 281 of {\em Grundlehren der
  Mathematischen Wissenschaften [Fundamental Principles of Mathematical
  Sciences]}.
\newblock Springer-Verlag, Berlin, 1985.

\bibitem[CMSS12]{CMSS}
Sylvain~E. Cappell, Laurentiu~G. Maxim, J{\"o}rg Sch{\"u}rmann, and Julius~L.
  Shaneson.
\newblock Equivariant characteristic classes of singular complex algebraic
  varieties.
\newblock {\em Comm. Pure Appl. Math.}, 65(12):1722--1769, 2012.

\bibitem[Coh80]{CohenReflections}
Arjeh~M. Cohen.
\newblock Finite quaternionic reflection groups.
\newblock {\em J. Algebra}, 64(2):293--324, 1980.

\bibitem[DBW14]{DBW_81res}
Maria Donten-Bury and Jaros{\l}aw~A. Wi\'{s}niewski.
\newblock On 81 symplectic resolutions of a 4-dimensional quotient by a group
  of order 32.
\newblock {\em arXiv:1409.4204 [math.AG]}, 2014.

\bibitem[Don69]{Don}
Peter Donovan.
\newblock The {L}efschetz-{R}iemann-{R}och formula.
\newblock {\em Bull. Soc. Math. France}, 97:257--273, 1969.

\bibitem[Ful98]{Fu2}
William Fulton.
\newblock {\em Intersection theory}, volume~2 of {\em Ergebnisse der Mathematik
  und ihrer Grenzgebiete. 3. Folge. A Series of Modern Surveys in Mathematics}.
\newblock Springer-Verlag, Berlin, second edition, 1998.

\bibitem[Hir56]{Hir}
F.~Hirzebruch.
\newblock {\em Neue topologische {M}ethoden in der algebraischen {G}eometrie}.
\newblock Ergebnisse der Mathematik und ihrer Grenzgebiete (N.F.), Heft 9.
  Springer-Verlag, Berlin-G\"ottingen-Heidelberg, 1956.

\bibitem[Huh16]{Huh}
June Huh.
\newblock Positivity of {C}hern classes of {S}chubert cells and varieties.
\newblock {\em J. Algebraic Geom.}, 25(1):177--199, 2016.

\bibitem[Huy05]{Huy}
Daniel Huybrechts.
\newblock {\em Complex geometry, An introduction}.
\newblock Universitext. Springer-Verlag, Berlin, 2005.

\bibitem[IR96]{ItoReid}
Yukari Ito and Miles Reid.
\newblock The {M}c{K}ay correspondence for finite subgroups of {${\rm
  SL}(3,\bold C)$}.
\newblock In {\em Higher-dimensional complex varieties ({T}rento, 1994)}, pages
  221--240. de Gruyter, Berlin, 1996.

\bibitem[Kal02]{KaledinMcKay}
D.~Kaledin.
\newblock Mc{K}ay correspondence for symplectic quotient singularities.
\newblock {\em Invent. Math.}, 148(1):151--175, 2002.

\bibitem[Lib15]{Lib}
Anatoly Libgober.
\newblock Elliptic genus of phases of {$N=2$} theories.
\newblock {\em Comm. Math. Phys.}, 340(3):939--958, 2015.

\bibitem[LS12]{LehnSorger}
Manfred Lehn and Christoph Sorger.
\newblock A symplectic resolution for the binary tetrahedral group.
\newblock In {\em Geometric methods in representation theory. {II}}, volume~24
  of {\em S\'emin. Congr.}, pages 429--435. Soc. Math. France, Paris, 2012.

\bibitem[McK80]{McKay}
John McKay.
\newblock Graphs, singularities, and finite groups.
\newblock In {\em The {S}anta {C}ruz {C}onference on {F}inite {G}roups ({U}niv.
  {C}alifornia, {S}anta {C}ruz, {C}alif., 1979)}, volume~37 of {\em Proc.
  Sympos. Pure Math.}, pages 183--186. Amer. Math. Soc., Providence, R.I.,
  1980.

\bibitem[MW15]{MiWe}
Ma{\l}gorzata {Mikosz} and Andrzej {Weber}.
\newblock {Equivariant Hirzebruch class for quadratic cones via degenerations.}
\newblock {\em {J. Singul.}}, 12:131--140, 2015.

\bibitem[Ohm06]{Oh}
Toru Ohmoto.
\newblock Equivariant {C}hern classes of singular algebraic varieties with
  group actions.
\newblock {\em Math. Proc. Cambridge Philos. Soc.}, 140(1):115--134, 2006.

\bibitem[PT07]{PeTu}
Andr{\'e}s Pedroza and Loring~W. Tu.
\newblock On the localization formula in equivariant cohomology.
\newblock {\em Topology Appl.}, 154(7):1493--1501, 2007.

\bibitem[Rei97]{ReidMcKay}
Miles Reid.
\newblock Mc{K}ay correspondence.
\newblock {\em http://arxiv.org/abs/alg-geom/9702016}, 1997.

\bibitem[Sai00]{Sai}
Morihiko Saito.
\newblock Mixed {H}odge complexes on algebraic varieties.
\newblock {\em Math. Ann.}, 316(2):283--331, 2000.

\bibitem[Ver00]{Verbitsky}
Misha Verbitsky.
\newblock Holomorphic symplectic geometry and orbifold singularities.
\newblock {\em Asian J. Math.}, 4(3):553--563, 2000.

\bibitem[Vey03]{Ve}
Willem Veys.
\newblock Stringy zeta functions for {$\Bbb Q$}-{G}orenstein varieties.
\newblock {\em Duke Math. J.}, 120(3):469--514, 2003.

\bibitem[Wae08a]{Wae1}
Robert Waelder.
\newblock Equivariant elliptic genera.
\newblock {\em Pacific J. Math.}, 235(2):345--377, 2008.

\bibitem[Wae08b]{Wae2}
Robert Waelder.
\newblock Equivariant elliptic genera and local {M}c{K}ay correspondences.
\newblock {\em Asian J. Math.}, 12(2):251--284, 2008.

\bibitem[Web12]{We1}
Andrzej Weber.
\newblock Equivariant {C}hern classes and localization theorem.
\newblock {\em J. Singul.}, 5:153--176, 2012.

\bibitem[Web16a]{We3}
Andrzej Weber.
\newblock Equivariant {H}irzebruch class for singular varieties.
\newblock {\em Selecta Math. (N.S.)}, 22(3):1413--1454, 2016.

\bibitem[Web16b]{Wbb}
Andrzej Weber.
\newblock Hirzebruch class and {B}ia{\l}ynicki-{B}irula decomposition.
\newblock {\em Trans. Groups}, 2016.
\newblock to appear, DOI 10.1007/s00031-016-9388-3,\hfill\break {\tt
  http://link.springer.com/article/10.1007/s00031-016-9388-3}.

\end{thebibliography}

\end{document}